\documentclass[a4paper,10pt]{amsart}

\usepackage{amsfonts,amssymb}
\newtheorem{proposition}{Proposition}[section]
\newtheorem{lemma}[proposition]{Lemma}
\newtheorem{corollary}[proposition]{Corollary}
\newtheorem{theorem}[proposition]{Theorem}
\theoremstyle{definition}

\theoremstyle{remark}
\newtheorem{remark}[proposition]{Remark}

\newcommand{\thlabel}[1]{\label{th:#1}}
\newcommand{\thref}[1]{Theorem~\ref{th:#1}}
\newcommand{\selabel}[1]{\label{se:#1}}
\newcommand{\seref}[1]{Section~\ref{se:#1}}
\newcommand{\lelabel}[1]{\label{le:#1}}
\newcommand{\leref}[1]{Lemma~\ref{le:#1}}
\newcommand{\prlabel}[1]{\label{pr:#1}}
\newcommand{\prref}[1]{Proposition~\ref{pr:#1}}
\newcommand{\colabel}[1]{\label{co:#1}}
\newcommand{\coref}[1]{Corollary~\ref{co:#1}}
\newcommand{\talabel}[1]{\label{ta:#1}}
\newcommand{\taref}[1]{Table~\ref{ta:#1}}

\newcommand{\relabel}[1]{\label{re:#1}}
\newcommand{\reref}[1]{Remark~\ref{re:#1}}

\newcommand{\eqlabel}[1]{\label{eq:#1}}
\newcommand{\equref}[1]{(\ref{eq:#1})}

\def\ra{\rightarrow}

\def\Id{{\rm Id}}
\def\im{{\rm Im}\,}
\def\Ker{{\rm Ker}\,}
\def\Ext{{\rm Ext}}
\def\Tor{{\rm Tor}}
\def\Hom{{\rm Hom}}

\def\NN{{\mathbb N}}

\def\ot{\otimes}
\def\va{\varepsilon}

\def\un{\underline}

\def\le{\langle}
\def\ri{\rangle}

\def\va{\varepsilon}
\def\v{\varphi}

\def\ra{\rightarrow}
\def\a{\alpha}
\def\b{\beta}

\def\d{\delta}

\def\ov{\overline}
\def\cal{\mathcal}
\def\un{\underline}

\def\units{{\mathbb G}_m}
\def\ZZ{{\mathbb Z}}

\newcommand{\Cc}{\cal C}

\newcommand{\Rr}{\cal R}
\def\equal#1{\smash{\mathop{=}\limits^{#1}}}

\allowdisplaybreaks[4]

\begin{document}
\title[The braidings of ${\rm Vect}^G$]
{The braided monoidal structures on the category of vector spaces graded by the Klein group}
\author{D. Bulacu}
\address{Faculty of Mathematics and Informatics, University
of Bucharest, Str. Academiei 14, RO-010014 Bucharest 1, Romania}
\email{daniel.bulacu@fmi.unibuc.ro}
\author{S. Caenepeel}
\address{Faculty of Engineering, 
Vrije Universiteit Brussel, B-1050 Brussels, Belgium}
\email{scaenepe@vub.ac.be}
\urladdr{http://homepages.vub.ac.be/\~{}scaenepe/}
\author{B. Torrecillas}
\address{Department of Algebra and Analysis\\
Universidad de Almer\'{\i}a\\
E-04071 Almer\'{\i}a, Spain}
\email{btorreci@ual.es}
\thanks{ 
This paper was written while the first author was visiting the Universidad 
de Almer\'{\i}a and the Vrije Universiteit Brussel. 
He would like to thank both universities for their warm hospitality. The first author was 
financially supported by FWO GP.045.09N and by CNCSIS $479/2009$, code ID 1904. 
This research was supported by the FWO project G.0117.10 ``Equivariant Brauer groups
and Galois deformations".}

\begin{abstract}
Let $k$ be a field, $k^*=k\setminus\{0\}$ and $C_2$ the cyclic group of order $2$. 
In this note we compute all the braided monoidal structures on the category of 
$k$-vector spaces graded by the Klein group $C_2\times C_2$. Actually, for the 
monoidal structures we will compute the explicit form of the $3$-cocycles on 
$C_2\times C_2$ with coefficients in $k^*$, 
while for the braided monoidal structures we will compute the explicit 
form of the abelian $3$-cocycles on $C_2\times C_2$ with coefficients in $k^*$. 
In particular, this will allow us to produce examples of quasi-Hopf algebras 
and weak braided Hopf algebras, out of the vector space $k[C_2\times C_2]$.    
\end{abstract}
\maketitle
\section{Introduction}\selabel{intro}
\setcounter{equation}{0}
For $k$ a field and $G$ a group it is well-known that the category of $G$-graded vector spaces
${\rm Vect}^G$ is a monoidal category. Now consider monoidal structures on ${\rm Vect}^G$,
with the same tensor product, unit object $k$, and the natural unit constraints, but with different
associativity constraints. These monoidal structures are in bijective correspondence
with the normalized $3$-cocycles on $G$ with  coefficients in $k^*=k\setminus \{0\}$.
In the case where $G$ is abelian,
braided monoidal structures on ${\rm Vect}^G$ are in bijective correspondence with
{\sl abelian} 3-cocycles, these are normalized 3-cocycles together with a so-called
R-matrix. Isomorphism classes of braided monoidal structures are then classified
by the cohomology group $H^3_{\rm ab}(G, k^*)$, which is isomorphic to the group of
quadratic forms $QF(G, k^*)$, by a result of Eilenberg and Mac Lane.

Associative algebras in ${\rm Vect}^G$ with one of these monoidal structures are usually
not associative in the usual sense; remarkable examples include Cayley-Dickson algebras
and Clifford algebras: in \cite{am0, am2}, Albuquerque and Majid show that they are
associative algebras in a suitable symmetric monoidal category of graded vector spaces.
Other examples are given by superalgebras, these are algebras in ${\rm Vect}^{C_2}$,
or, more generally, associative algebras in the category of anyonic vector spaces ${\rm Vect}^{\rm C_n}$,
with a suitable monoidal structure.

In the case where $G$ is a cyclic group,
the classification of braided monoidal structures on ${\rm Vect}^{G}$ is presented
in \cite{js, jsp}. The monoidal structures on ${\rm Vect}^{\mathbb{Z}}$ are all trivial, and the
braided monoidal structures are given by the $R$-matrices
${\cal R}_\a: \mathbb{Z}\times \mathbb{Z}\ra k^*$, 
${\cal R}_\a(x, y)=\a^{xy}$, with $\alpha\in k^*$.  
For a finite cyclic group, we have
$H^3(\mathbb{Z}_n, k^*)\cong \mu_n(k)$, the group of $n$-th roots of $1$. The cohomology class corresponding
to $q\in \mu_n(k)$ is represented by the normalized $3$-cocycle
\begin{equation}\eqlabel{3coccyc}
\phi_q(x, y, z)=\left\{
\begin{array}{lr}
1&\mbox{if $y+z<n$}\\
q^x&\mbox{if $y+z\geq n$}
\end{array} \right. \hspace{2mm},
\mbox{for all $x, y, z\in \{0, 1, \cdots , n - 1\}$.}
\end{equation}
The braidings on ${\rm Vect}^{\mathbb{Z}_n}$ are represented by abelian cocycles 
$(\phi_{\nu^n}, {\cal R}_\nu)$ 
with $\nu\in k^*$ such that $\nu^{n^2}=\nu^{2n}=1$, 
where  
${\cal R}(x, y)=\nu^{xy}$, for all $x, y\in \{0, \cdots, n-1\}$. Notice that 
the case $G=\mathbb{Z}_3$ has been handled also in \cite[Prop. 6 and 7]{am}.

In this note, we continue the classification of the braided monoidal structures on ${\rm Vect}^G$. 
As we have explained, the cyclic case has been covered completely; in what follows we complete the
classification in the easiest remaining case, that is the case where $G$ is the
Vierergruppe von Klein, the product $C_2\times C_2$ of two cyclic groups of order 2.
Using techniques from homological 
algebras we can describe $H^3(C_2\times C_2, k^*)$, see \prref{A18}. But for the explicit 
monoidal structures we need explicit formulas for the representing cocycles.
In \seref{2},we reduce this problem to the computation of the so-called happy $3$-cocycles and then, 
after some computations, we were able to find out the explicit form of 
these elements, see \thref{A10}. It came out that there are three types of 
normalized happy $3$-cocycles which were denoted by $\phi_X$, $h_a$ 
and $g_b$, respectively ($X\subseteq (C_2\times C_2)\backslash\{e\}$, and $a, b\in k^*$). 

The abelian cocycles are computed in \seref{braided}. 
Using the Eilenberg-Mac Lane Theorem we can compute $H^3_{\rm ab}(C_2\times C_2, k^*)$, 
but, again, we want to compute the explicit formulas for the R-matrices.
When $k$ does 
not contain a primitive fourth root of unit there are $8$ non-isomorphic braidings on 
${\rm Vect}^{C_2\times C_2}$, all of them having the trivial cocycle as a underlying 
$3$-cocycle, cf. \prref{ab3cocycfromtriv}. If $k$ has a primitive fourth root of unit $i$,
then we have $24$ additional non-isomorphic braidings, with underlying cocycles
$\phi_X$, with $|X|=2$, see \prref{ab3cocycfromnontriv}. In this case
$H^3_{\rm ab}(C_2\times C_2, k^*)\cong C_4\times C_4\times C_2$. In both cases,
there are only four non-isomorphic symmetric monoidal structures.

The importance of having the explicit formula for the cocycles and R-matrices is
illustrated in the final \seref{appl}, where we construct explicit examples of quasi-Hopf 
algebras and weak braided Hopf algebras. Quasi-Hopf algebras 
are obtained by determining explicitely some Harrison $3$-cocycles corresponding to 
some $3$-cocycles on $C_n$ and $C_2\times C_2$, while the weak braided Hopf algebras 
are build on $k[C_n]$ or $k[C_2\times C_2]$ with the help of a coboundary 
abelian $3$-cocycle on $C_n$, respectively on $C_2\times C_2$. Note that the latter
are commutative and cocommutative weak braided Hopf algebras in 
${\rm Vect}^{C_n}$ or ${\rm Vect}^{C_2\times C_2}$, which are symmetric 
monoidal categories with the braided monoidal structures defined by a so-called 
$2$-cochain on $C_n$ (resp. on $C_2\times C_2$).

\section{Preliminary results}\selabel{prelim}
\setcounter{equation}{0}
\subsection{Braided monoidal categories}\selabel{1.1}
A monoidal category is a category $\Cc$ together with 
a functor $\ot:\ \Cc\times\Cc\to \Cc$, called the tensor product, an
object $\un{1}\in \Cc$ called the unit object, and natural isomorphisms 
$a:\ \ot\circ (\ot\times {\rm Id})\to \ot\circ ({\rm Id}\times \ot)$ (the associativity
constraint), $l:\ \ot\circ (\un{1}\times {\rm Id})\to \Id$ (the left unit constraint) and 
$r:\ \ot\circ ({\rm Id}\times \un{1})\to \Id$ (the right unit constraint). In addition, 
$a$ has to satisfy the pentagon axiom, and $l$ and $r$ have to satisfy 
the triangle axiom. We refer to \cite{k, maj} for a detailed discussion. In the sequel,
for any object $X\in \Cc$ we will identify 
$\un{1}\ot X\cong X\cong X\ot\un{1}$ using $l_X$ and $r_X$. 

The switch functor $\tau:\ {\cal C}\times{\cal C}\ra \Cc\times \Cc$ 
is defined by $\tau(X, Y)=(Y, X)$. 
A braiding on a monoidal category $\Cc$ 
is a natural isomorphism $c:\ \ot\to\ot\circ\tau$, satisfying 
certain axioms, see for example \cite{k, maj}. If 
$c_{X, Y}=c^{-1}_{Y, X}$, for all $X, Y\in \Cc$, then we call $\Cc$ a 
symmetric monoidal category.

We are mainly interested in the case where $\Cc= {\rm Vect}^G$, the category of vector spaces
graded by a group $G$. We will write $G$ multiplicatively, and denote $e$ for the unit element
of $G$. Recall that a $G$-graded vector space is a vector space $V$ together with a direct sum
decomposition $V=\bigoplus\limits_{g\in G}V_g$. An element $v\in V_g$ is called
homogeneous of degree $g$, and we write $|v|=g\in G$. For two $G$-graded vector spaces
$V$ and $W$, a $k$-linear map $f:\ V\to W$ is said to preserve the grading if
$f(V_g)\subseteq W_g$, for all $g\in G$. ${\rm Vect}^G$ is the category of $G$-graded 
vector spaces and grade preserving $k$-linear maps.

\subsection{Monoidal structures on ${\rm Vect}^G$}\selabel{1.2}
For two $G$-graded vector spaces $V$ and $W$, $V\ot W$ is again a $G$-graded vector
space, with $(V\ot W)_g:=\bigoplus\limits_{\sigma\tau=g}^{}V_\sigma\ot W_\tau$.
If the $k$-linear maps $f:V\ra V'$ and $g: W\ra W'$ 
are grade preserving, then $f\ot g$ is also grade preserving, hence we have a functor
$\ot:\ {\rm Vect}^G\times {\rm Vect}^G\to {\rm Vect}^G$. $k$ is a $G$-graded vector space,
with $k_e=k$ and $k_g=0$, for all $G\ni g\not=e$. The problem is now to describe
monoidal structures on ${\rm Vect}^G$, with $\ot$ as tensor product and $k$ as the
unit object. It is known that these correspond to $3$-cocycles $\phi$ on $G$ with coefficients 
in $k^*$. To solve this problem, we need to give the possible associativity and unit constraints.

To this end, let us first recall the definition of group cohomology.
Let $K^n(G, k^*)$ be the set of maps from $G^n$ 
to $k^*$. $K^n(G, k^*)$ is a group under pointwise multiplication. We have maps
$\Delta_n:\ K^n(G, k^*)\to K^{n+1}(G, k^*)$. $\Delta_2$ and $\Delta_3$ are given by
the formulas
\begin{eqnarray*}
&&\Delta_2(g)(x, y, z)=g(y, z)g(xy, z)^{-1}g(x, yz)g(x, y)^{-1};\\
&&\Delta_3(f)(x, y, z, t)=
f(y, z, t)f(xy, z, t)^{-1}f(x, yz, t)f(x, y, zt)^{-1} f(x, y, z).
\end{eqnarray*}
It is well-known that $B^n(G,k^*)=\im\Delta_{n-1}\subseteq Z^n(G,k^*)=\Ker(\Delta_n)$.
The $n$-th cohomology group is defined as $H^n(G,k^*)=Z^n(G,k^*)/B^n(G,k^*)$, 
and two elements of $H^n(G, k^*)$ are called cohomologous if they lie in 
the same equivalence class. 
The elements of $Z^3(G,k^*)$ are called 3-cocycles, and the elements of
$B^3(G,k^*)$ are called 3-coboundaries. 
$\phi\in K^3(G,k^*)$ is a 3-cocycle if and only if
\begin{equation}\eqlabel{A1.1}
\phi(y, z, t)\phi(x, yz, t)\phi(x, y, z)=\phi(xy, z, t)\phi(x, y, zt),
\end{equation}
for all $x, y, z\in G$. A 3-cocycle $\phi$ is called {\sl normalized} if
$\phi(x,1,z)=1$, for all $x,z\in G$.

\begin{lemma}\lelabel{A1}
If $\phi$ is a normalized 3-cocycle, then $\phi(e, y, z)=\phi(x, y, e)=1$, for all $x, y, z\in G$.
A coboundary $\Delta_2(\psi)$ is normalized if and only if $\psi(e, x)=\psi(z, e)$, 
for all $x, z\in G$. Then $\psi$ is called a normalized 2-cochain on $G$.
\end{lemma}

\begin{proof}
Taking $z=e$ in \equref{A1.1}, we find that $\phi(x, y, e)=1$. Taking $y=e$, we find that
$\phi(e, z, t)=1$. The proof of the second statement is straightforward.
\end{proof}

Monoidal structures on ${\rm Vect}^G$ are in bijective correspondence to 
the elements of $Z^3(G,k^*)$. Given a $3$-cocycle $\phi$, the corresponding associativity
constraint is given by the formula
$$a_{V, W, Z}((v\ot w)\ot z)=\phi(| v|,| w|, |z|)v\ot (w\ot z),$$
for $G$-graded vector spaces $V,W$ and $Z$, and
$v\in V$, $w\in W$ and $z\in Z$ homogeneous. The unit constraints are given by
the formulas
$$r_V(v\ot 1)=\phi(|v|,e,e)v~~;~~l_V(1\ot v)=\phi(e,e,|v|)^{-1}v.$$
If $\phi$ is normalized, then the unit constraints are trivial. Two monoidal
structures on ${\rm Vect}^G$ give isomorphic monoidal categories if and only if
their corresponding $3$-cocycles are cohomologous. In order to solve our problem,
it therefore suffices to compute the $3$-cocycles that represent the elements of
$H^3(G,k^*)$. Actually, we can restrict attention to normalized cocycles.

\begin{lemma}\lelabel{A3}
Every 3-cocycle $\phi$ is cohomologous to a normalized 3-cocycle.
\end{lemma}

\begin{proof}
Take $y=z=1$ in \equref{A1.1}. Then we find 
$\phi(x, 1, t)=\phi(1, 1, t)\phi(x, 1, 1)$. In particular, it follows that 
$\phi(1, 1, 1)=1$ (take $x=t=1$). 
Then consider the map $g:\ G\times G\to k^*$, 
$g(x, y)=\phi(1, 1, y)^{-1}\phi(x, 1, 1)$, and compute that
\begin{eqnarray*}
\Delta_2(g)(x, 1, y)&=&g(1, y)g(x, y)^{-1}g(x, y)g(x, 1)^{-1}\\
&=&\phi(1, 1, y)^{-1}\phi(1, 1, 1)\phi(1, 1, 1)\phi(x, 1, 1)^{-1}=\phi(x, 1, y)^{-1}.
\end{eqnarray*}
It then follows that $\phi\Delta_2(g)$ is normalized.
\end{proof}

Let $B^3_n(G, k^*)$ and $Z^3_n(G, k^*)$ be the subgroups of $B^3(G, k^*)$ 
and $Z^3(G, k^*)$ consisting of normalized elements. We then have a well-defined 
group morphism 
\[
Z^3_n(G, k^*)/B^3_n(G, k^*)\ni \hat{\phi}\mapsto \ov{\phi}\in Z^3(G, k^*)/B^3(G, k^*)
\] 
which is surjective by \leref{A3}. It is easy to see that it is also injective, 
and therefore   
$$
H^3(G, k^*)= Z^3_n(G, k^*)/B^3_n(G, k^*).
$$

\subsection{Braided monoidal structures on ${\rm Vect}^G$}\selabel{1.3}
The next problem is to describe all {\sl braided} monoidal structures on ${\rm Vect}^G$,
with $\ot$ as tensor product and $k$ as unit object. Such a braiding can only exist in the
case where $G$ is abelian. In this case, these structures are in
bijective correspondence with so-called {\sl abelian} $3$-cocycles in $G$,
see \cite{js, jsp}. An abelian $3$-cocycle is a pair $(\phi, \Rr)$, where $\phi$ is a
normalized $3$-cocycle, and $\Rr:\ G\times G\to k^*$ is a map satisfying
\begin{eqnarray}
&&{\cal R}(xy, z)\phi(x, z, y)=\phi(x, y, z){\cal R}(x, z)\phi(z, x, y)
{\cal R}(y, z);\eqlabel{r1coc}\\
&&\phi(x, y, z){\cal R}(x, yz)\phi(y, z, x)={\cal R}(x, y)
\phi(y, x, z){\cal R}(x, z)\eqlabel{r2coc},
\end{eqnarray}
for all $x, y, z\in G$. If we take $x=y=e$ in \equref{r1coc} and
$y=z=e$ in \equref{r1coc}, then we obtain immediately that
\begin{equation}\eqlabel{r0coc}
\Rr(e,z)=\Rr(x,e)=1.
\end{equation}
We call $\phi$ the underlying $3$-cocycle, and ${\cal R}$ the R-matrix.
The corresponding monoidal structure is defined by $\phi$,
and the braiding is described by the formula
$$c_{V, W}(v\ot w)={\cal R}(|v|,|w|)w\ot v,$$
for all $V, W\in {\rm Vect}^G$ and $v\in V$ and $w\in W$ homogeneous. It is easy
to show that the formulas (\ref{eq:r1coc}-\ref{eq:r2coc}) express the commutativity of
the hexagonal diagrams in the definition of a braiding.

Take $\psi:\ G\times G\to k^*$ satisfying $\psi(e, x)=\psi(y, e)$, for all $x,y\in G$, so
that $\Delta_2(\psi)$ is normalized. Consider the map
${\cal R}_\psi: G\times G\ra k^*$, ${\cal R}_\psi(x, y)=\psi(x, y)^{-1}\psi(y, x)$.
Then $(\Delta_2(\psi),\Rr_{\psi})$ is an abelian 3-cocycle, called an abelian
$3$-coboundary. The sets $Z^3_{\rm ab}(G,k^*)$ of abelian 3-cocycles
and $B^3_{\rm ab}(G,k^*)$ of abelian 3-coboundaries are abelian groups under
pointwise multiplication, and $H^3_{\rm ab}(G,k^*)$ is defined as the quotient
$Z^3_{\rm ab}(G,k^*)/B^3_{\rm ab}(G,k^*)$. Two braided monoidal structures on
${\rm Vect}^G$ are isomorphic if and only if their corresponding abelian $3$-cocycles
are cohomologous. Finally, observe that $(\phi, {\cal R})\in Z^3_{\rm ab}(G, k^*)$ defines 
a {\sl symmetric} monoidal 
structure on ${\rm Vect}^G$ if and only if ${\cal R}(x, y){\cal R}(y, x)=1$, 
for all $x, y\in G$.

\subsection{Some homological algebra}\selabel{1.4}
We can compute certain cohomology groups using techniques from homological algebra. 
For $r,s\in \NN_0$, let $(r,s)$ be the 
greatest common divisor of $r$ and $s$. Denote by $\mu_r(k)$ the group of $r$-th units in $k$, 
and by $k^{*(r)}:=\{\a^r\mid \a\in k\}$. We also denote by 
$C_s$ the cyclic group of order $s$ written multiplicatively and by 
$\mathbb{Z}_r$ the cyclic group of order $r$ written, this time, additively.  

\begin{proposition}\prlabel{A18}
Let $k$ be a field, and $r,s\in \NN_0$. Then
\begin{equation}\eqlabel{A18.0}
H^3(C_r, k^*)=\mu_r(k)~~;~~H^3(C_r\times C_s)=
k^*/k^{*(r,s)}\times \mu_r(k)\times \mu_s(k)\times \mu_{(r,s)}(k).
\end{equation}
\end{proposition}

\begin{proof}
We have the following consequence of the Universal Coefficient Theorem, see
for example \cite[Exercise 6.1.5]{weibel}:
\begin{equation}\eqlabel{A18.1}
H^n(G, k^*)\cong \Ext^1_{\ZZ}(H_{n-1}(G, \ZZ),k^*)\times \Hom_{\ZZ}(H_n(G, \ZZ),k^*),
\end{equation}
where $H_n(G, \ZZ)=\Tor^{\ZZ G}_n(\ZZ, \ZZ)$ is the $n$-th homology group of $G$ with
values in $\ZZ$.

1) We apply \equref{A18.1} in the case where $n=3$ and
$G=C_r\cong \ZZ_r$. From \cite[Example 6.2.3]{weibel},
we recall that
$$
H_0(C_r,\ZZ)=\ZZ,~~H_{2n-1}(C_r, \ZZ)=\ZZ_r,~~H_{2n}(C_r, \ZZ)=0,
$$
for $n\geq 1$. It follows immediately that
$\Ext^1_{\ZZ}(H_{2}(C_r, \ZZ), k^*)=0$, and
$$
\Hom_{\ZZ}(H_3(C_r, \ZZ),k^*)=\Hom_{\ZZ}(C_r, k^*)=\mu_r(k), 
$$
proving the first formula in \equref{A18.0}.

2) It follows from the K\"unneth formula, see for example 
\cite[Prop. 6.1.13]{weibel}, that
\begin{eqnarray*}
&&\hspace*{-3cm}
H_n(G\times H, \ZZ)
\cong
\prod_{p+q=n} H_p(G, \ZZ)\otimes_{\ZZ} H_q(H, \ZZ) \\
&\times&
\prod_{p+q=n-1} \Tor^{\ZZ}_1(H_p(G, \ZZ),H_q(H, \ZZ)).
\end{eqnarray*}
Now, for any positive integers $p, q$ we have  
$$
H_p(C_r, \ZZ)\otimes_{\ZZ} H_q(C_s, \ZZ)= 
\begin{cases}
\ZZ_r\ot_{\ZZ}\ZZ_s\cong \ZZ_{(r,s)}&\mbox{\rm for $p, q$ odd}\\
\ZZ\ot_{\ZZ}\ZZ_s\cong \ZZ_s&\mbox{\rm for $p=0$ and $q$ odd}\\
\ZZ_r\ot_{\ZZ}\ZZ\cong \ZZ_r&\mbox{\rm for $p$ odd and $q=0$}\\
0&\mbox{\rm otherwise}
\end{cases}\hspace{2mm},
$$ 
and
$$
\Tor^{\ZZ}_1(H_p(C_r, \ZZ),H_q(C_s, \ZZ))=
\begin{cases}
\Tor^{\ZZ}_1(\ZZ_r, \ZZ_s)=\ZZ_{(r,s)}&\mbox{\rm for $p, q$ odd}\\
\Tor^{\ZZ}_1(\ZZ, \ZZ_s)=0&\mbox{\rm for $p=0$ and $q$ odd}\\
\Tor^{\ZZ}_1(\ZZ_r,\ZZ)=0&\mbox{\rm for $p$ odd and $q=0$}\\
0&\mbox{\rm otherwise}
\end{cases}\hspace{2mm}.
$$
Substituting these formulas in the K\"unneth formula, we find that 
$$
H_2(C_r\times C_s, \ZZ)= \ZZ_{(r,s)}~{\rm and}~
H_3(C_r\times C_s, \ZZ)=\ZZ_r\times \ZZ_s\times \ZZ_{(r,s)}.
$$
It is well-known that 
$$
\Ext^1_{\ZZ}(\ZZ_{(r,s)}, k^*)=k^*/k^{*(r,s)}
$$
(see for example \cite[Theorem 7.17]{rotman}), and that 
$$
\Hom_{\ZZ}(\ZZ_r\times \ZZ_s\times \ZZ_{(r,s)}, k^*)=
\mu_r(k)\times \mu_s(k)\times \mu_{(r,s)}(k).
$$
The second formula in \equref{A18.0} then follows after we apply \equref{A18.1}.
\end{proof}

\subsection{The Eilenberg-Mac Lane Theorem}\selabel{1.5}
The Eilenberg-Mac Lane Theorem gives a description of $H^3_{\rm ab}(G, k^*)$ 
for an arbitrary abelian group $G$.
A function $Q: G\ra k^*$ between abelian groups $G$, $k^*$ is called 
a quadratic form when $Q(x^{-1})=Q(x)$, and  
\begin{equation}\eqlabel{quadraticform}
Q(xyz)Q(x)Q(y)Q(z)=Q(xy)Q(xz)Q(yz),
\end{equation}  
for all $x, y, z\in G$. The set of quadratic forms on $G$ with values in $k^*$ is 
denoted by $QF(G, k^*)$. It is easy to see that the pointwise product of two 
quadratic forms is again a quadratic form, so $QF(G, k^*)$ 
is an abelian group.

\begin{theorem}\thlabel{EM} {\bf (Eilenberg-Mac Lane \cite{eml1, smcl1})}
Let $G$ be an abelian group and $(\phi, {\cal R})\in H^3_{\rm ab}(G, k^*)$. Then $Q: G\ra k^*$ given by 
$Q(x)={\cal R}(x, x)$, for all $x\in G$, is a quadratic form on $G$ with values in $k^*$. It is called 
the trace of the abelian $3$-cocycle $(\phi, {\cal R})$. Furthermore, trace induces a group isomorphism 
$EM: H^3_{\rm ab}(G, k^*)\ra QF(G, k^*)$. 
\end{theorem}

For a proof of the Eilenberg-Mac Lane Theorem, we refer to \cite[p. 35, Theorem 12]{js}.

\section{Computation of the $3$-cocycles on the Vierergruppe}\selabel{2}
\setcounter{equation}{0}
The {\sl Vierergruppe von Klein} is the non-cyclic group of order four, $C_2\times C_2$.
It follows from \prref{A18} that
$$H^3(C_2\times C_2, k^*)\cong k^*/k^{*(2)}\times \mu_2(k)\times \mu_2(k).$$
In order to be able to describe the monoidal structures on the category of vector spaces
by the Klein group, we need the cocycles explicitly. This explicit form will also be needed
in \seref{braided}, where we will compute the abelian $3$-cocycles.
In this Section, we will compute the cocycles manually.

We will work over a field $k$ of characteristic different from 2. 
In the sequel, we will write 
$G=C_2\times C_2=\{e,\sigma,\tau,\rho\}$, with $\sigma\tau=\tau\sigma=\rho$ 
and $\sigma^2=\tau^2=e$. 

\subsection{The normalized coboundaries}\selabel{2.1}
Consider $g:\ C_2\times C_2\to k^*$. If $\Delta_2(g)$ is a normalized 
coboundary, then $g$ is determined completely by the following data
(see \leref{A1}):
\begin{equation}\eqlabel{A5.1}
\begin{matrix}
g(\sigma, \sigma)=a_1&,\hspace{2mm}g(\tau, \tau)=a_2&,\hspace{2mm}g(\rho, \rho)=a_3\hspace{1mm},\cr
g(\sigma, \tau)=b_1&,\hspace{2mm}g(\tau, \rho)=b_2 &,\hspace{2mm}g(\rho, \sigma)=b_3\hspace{1mm},\cr
g(\tau, \sigma)=b_4&,\hspace{2mm}g(\sigma, \rho)=b_5&,\hspace{2mm}g(\rho, \tau)=b_6\hspace{1mm}.\cr
g(x,1)=c&,\hspace{2mm}g(1,x)=c,&\cr
\end{matrix}
\end{equation}
for all $x\in G$.
All other values of $g$ are equal to $1$. Normalized coboundaries thus depend on the choice
of ten parameters $a_1,a_2,a_3,b_1,b_2,b_3,b_4,b_5,b_6,c\in k^*$.
For later use, we list some of the values of
$\Delta_2(g)$:
\begin{equation}\eqlabel{A5.2a}
\Delta_2(g)(\sigma, \sigma, \sigma)
=\Delta_2(g)(\tau, \tau, \tau)=\Delta_2(g)(\rho, \rho, \rho)=1.
\end{equation}
\begin{equation}\eqlabel{A5.2}
\begin{matrix}
\Delta_2(g)(\sigma,\sigma,\tau)=b_1b_5a_1^{-1}c^{-1}&,\hspace{2mm}
\Delta_2(g)(\rho,\rho,\sigma)=b_2b_6a_3^{-1}c^{-1}\hspace{1mm},\cr
\Delta_2(g)(\tau,\sigma,\sigma)=b_4^{-1}b_3^{-1}a_1c&,\hspace{2mm} 
\Delta_2(g)(\sigma,\tau,\sigma)=b_1^{-1}b_3^{-1}b_4b_5\hspace{1mm},\cr
\Delta_2(g)(\sigma,\tau,\rho)=b_1^{-1}b_2a_1a_3^{-1}&,\hspace{2mm}
\Delta_2(g)(\tau,\rho,\sigma)=b_2^{-1}b_3a_2a_1^{-1}\hspace{1mm},\cr
\Delta_2(g)(\rho,\sigma,\tau)=b_3^{-1}b_1a_3a_2^{-1}&,\hspace{2mm}
\Delta_2(g)(\tau,\sigma,\rho)=b_4^{-1}b_5a_2a_3^{-1}\hspace{1mm},\cr
\Delta_2(g)(\sigma,\rho,\tau)=b_5^{-1}b_6a_1a_2^{-1}&,\hspace{2mm}
\Delta_2(g)(\rho,\tau,\sigma)=b_6^{-1}b_4a_3a_1^{-1}\hspace{1mm}.\cr
\end{matrix}
\end{equation} 

\subsection{The cocycle relations}\selabel{2.2}
Taking $x=y=z=t=\sigma$ in \equref{A1.1}, we find that $\phi(\sigma, \sigma, \sigma)^2=1$. Thus
\begin{equation}\eqlabel{A6.0}
\varepsilon_\sigma=\phi(\sigma, \sigma, \sigma)=\pm 1.
\end{equation}
We have similar formulas for $\varepsilon_\tau=\phi(\tau, \tau, \tau)$ and 
$\varepsilon_\rho=\phi(\rho, \rho, \rho)$. Since every coboundary takes the value 1 at 
$(\sigma, \sigma, \sigma)$, $(\tau, \tau, \tau)$ and $(\rho, \rho, \rho)$, we see that 
$\varepsilon_\sigma, \varepsilon_\tau$ and $\varepsilon_\rho$  stay invariant if 
we replace $\phi$ by a cohomologous cocycle.

\begin{lemma}\lelabel{A6}
Any normalized 3-cocycle $\phi$ satisfies the relations
\begin{eqnarray}
&&\phi(\rho, \tau, \tau)=\varepsilon_\tau \phi(\sigma, \tau, \tau)\hspace{1mm},\eqlabel{A6.1}\\
&&\phi(\tau, \tau, \rho)=\varepsilon_\tau \phi(\tau, \tau, \sigma)\hspace{1mm},\eqlabel{A6.2}\\
&&\phi(\tau, \rho, \tau)\phi(\tau, \sigma, \tau)=\varepsilon_\tau \hspace{1mm},\eqlabel{A6.3}\\
&&\phi(\sigma, \tau, \tau)\phi(\sigma, \sigma, \tau)\phi(\sigma, \rho, \tau)=1\hspace{1mm},\eqlabel{A6.4}\\
&&\phi(\tau, \sigma, \tau)\phi(\sigma, \tau, \sigma)\phi(\sigma, \rho, \tau)=\phi(\rho, \sigma, \tau)
\phi(\sigma, \tau, \rho)\hspace{1mm},\eqlabel{A6.5}\\
&&\phi(\sigma, \tau, \tau)\phi(\tau, \tau, \sigma)
=\phi(\rho, \tau, \sigma)\phi(\sigma, \tau, \rho)\hspace{1mm},\eqlabel{A6.6}\\
&&\varepsilon_\sigma  \phi(\tau, \rho, \sigma)\phi(\sigma, \tau, \rho)=\phi(\rho, \rho, \sigma)
\phi(\sigma, \tau, \tau)\hspace{1mm},\eqlabel{A6.7}\\
&&\phi(\tau, \rho, \tau)\phi(\sigma, \sigma, \tau)
\phi(\sigma, \tau, \rho)=\phi(\rho, \rho, \tau)\phi(\sigma, \tau, \sigma)\hspace{1mm},
\eqlabel{A6.8}\\
&&\phi(\tau, \rho, \rho)\phi(\sigma, \sigma, \rho)\phi(\sigma, \tau, \rho)
=\varepsilon_\rho \hspace{1mm}.\eqlabel{A6.9}
\end{eqnarray}
These relations remain valid after we apply a permutation of $(\sigma,\tau,\rho)$.
\end{lemma}

\begin{proof}
All the formulas follow directly from the cocycle relation \equref{A1.1}: we subsequently take
$x=\sigma$, $y=z=t=\tau$ \equref{A6.1}; 
$x=y=z=\tau$, $t=\sigma$ \equref{A6.2}; 
$y=\sigma$, $x=z=t=\tau$ \equref{A6.3} (applying \equref{A6.1}); 
$x=y=\sigma$, $z=t=\tau$ \equref{A6.4};
$x=z=\sigma$, $y=t=\tau$ \equref{A6.5}; 
$x=t=\sigma$, $y=z=\tau$ \equref{A6.6};
$x=t=\sigma$, $y=\tau$, $z=\rho$ \equref{A6.7};
$x=\sigma$, $y=t=\tau$, $z=\rho$ \equref{A6.8};
$x=\sigma$, $y=\tau$, $z=t=\rho$ \equref{A6.9}.
\end{proof}

\begin{lemma}\lelabel{A7}
Let $\phi$ be a normalized cocycle, and write
$$
p=\phi(\sigma,\tau,\rho)\phi(\tau,\rho,\sigma)\phi(\rho,\sigma,\tau)~~;~~
q=\phi(\rho,\tau,\sigma)\phi(\sigma,\rho,\tau)\phi(\tau,\sigma,\rho).
$$
Then
\begin{equation}\eqlabel{A7.1}
p=q=\varepsilon_\sigma\varepsilon_\tau\varepsilon_\rho=\pm 1.
\end{equation}
\end{lemma}

\begin{proof}
We compute:
\begin{eqnarray*}
p&=& \phi(\sigma,\tau,\rho)\phi(\tau,\rho,\sigma)\phi(\rho,\sigma,\tau)
\equal{\equref{A6.7}} \varepsilon_\sigma \phi(\rho, \rho, \sigma)
\phi(\sigma, \tau, \tau)\phi(\rho, \sigma, \tau)\\
&\equal{(\ref{eq:A6.1},\ref{eq:A6.2})}&
\varepsilon_\sigma\varepsilon_\rho\varepsilon_\tau 
\phi(\rho, \rho, \tau)\phi(\rho, \tau, \tau)\phi(\rho, \sigma, \tau)\equal{\equref{A6.4}}
\varepsilon_\sigma\varepsilon_\tau \varepsilon_\rho.
\end{eqnarray*}
In a similar way, we prove that $q= \varepsilon_\sigma\varepsilon_\tau \varepsilon_\rho$.
\end{proof}

\subsection{Reduction to happy cocycles}\selabel{2.3}
A normalized 3-cocycle $\phi$ is called 
\index{even cocycle}
even, respectively 
\index{odd cocycle}
odd if $p=1$, respectively $p=-1$. 
$\phi$ is called 
\index{happy cocycle}
happy if $\phi(x,y,z)=p$ whenever $x$, $y$ and $z$ are pairwise distinct and 
different from $1$.

\begin{proposition}\prlabel{A8}
Every 3-cocycle $\phi$ is cohomologous to a happy normalized 3-cocycle.
\end{proposition}
\begin{proof}
It follows from \leref{A3} that we can assume that $\phi$ is normalized.
Let $g$ be defined as in \equref{A5.1}, with $a_1=a_2=a_3=1$,
$b_1=b_5=p$, $b_2=\phi(\sigma, \tau, \rho)^{-1}$, $b_3=\phi(\rho, \sigma, \tau)$, 
$b_4=\phi(\tau, \sigma, \rho)$, $b_6=\phi(\sigma, \rho, \tau)^{-1}$ and $c=1$. 
Applying \equref{A5.2}, we find immediately that $\phi\Delta_2(g)$ is happy.
\end{proof}
\subsection{Description of the happy cocycles}\selabel{2.4}
Assume that $\phi:\ G\times G\times G\to k^*$ is happy and normalized. This means that it 
satisfies the following properties:
\begin{enumerate}
\item $\phi(x, y, z)=1$, if one of the three entries is $1$;
\item $\varepsilon_x=\phi(x, x, x)=\pm 1$, for all $x\in \{\sigma, \tau, \rho\}$;
\item $\phi(x, y, z)=p=\varepsilon_\sigma\varepsilon_\tau \varepsilon_\rho$, if
$(x, y, z)$ is a permutation of $(\sigma, \tau, \rho)$.
\end{enumerate}
The cocycle relations (\ref{eq:A6.4}-\ref{eq:A6.9}) then simplify as follows:
\begin{eqnarray}
&&\phi(\sigma, \tau, \tau)\phi(\sigma, \sigma, \tau)=p\hspace{1mm},\eqlabel{A9.4}\\
&&\phi(\tau, \sigma, \tau)\phi(\sigma, \tau, \sigma)=p\hspace{1mm},\eqlabel{A9.5}\\
&&\phi(\sigma, \tau, \tau)\phi(\tau, \tau, \sigma)=1\hspace{1mm},\eqlabel{A9.6}\\
&&\varepsilon_\sigma =\phi(\rho, \rho, \sigma)
\phi(\sigma, \tau, \tau)\hspace{1mm},\eqlabel{A9.7}\\
&&p\phi(\tau, \rho, \tau)\phi(\sigma, \sigma, \tau) =\phi(\rho, \rho, \tau)
\phi(\sigma, \tau, \sigma)\hspace{1mm},\eqlabel{A9.8}\\
&&\phi(\tau, \rho, \rho)\phi(\sigma, \sigma, \rho)=p\varepsilon_\rho\hspace{1mm}.\eqlabel{A9.9}
\end{eqnarray}

\begin{proposition}\prlabel{A9}
Let $\phi:\ G\times G\times G\to k^*$ be happy and normalized. Then $\phi$ is a 3-cocycle
if and only if  (\ref{eq:A6.1}-\ref{eq:A6.3}) and (\ref{eq:A9.4}-\ref{eq:A9.6}) (and their
permuted versions) are satisfied. In other words, the cocycle relations 
(\ref{eq:A9.7}-\ref{eq:A9.9}) follow from the other cocycle relations.
\end{proposition}
\begin{proof}
One implication is clear. Conversely, suppose that (\ref{eq:A6.1}-\ref{eq:A6.3}) and 
(\ref{eq:A9.4}-\ref{eq:A9.6}) are satisfied. We show that (\ref{eq:A9.7}-\ref{eq:A9.9}) hold. 
Indeed, we have 
\begin{eqnarray*}
\phi(\rho, \rho, \sigma) \phi(\sigma, \tau, \tau)&\equal{(\ref{eq:A6.1},\ref{eq:A6.2})}&
\varepsilon_\rho \phi(\rho, \rho, \tau) \varepsilon_\tau \phi(\rho, \tau, \tau)
\equal{\equref{A9.4}} \varepsilon_\rho\varepsilon_\tau p=\varepsilon_\sigma;\\
p\phi(\tau, \rho, \tau)\phi(\sigma, \sigma, \tau)&\equal{(\ref{eq:A6.3},\ref{eq:A9.4})}&
p\varepsilon_\tau \phi(\tau, \sigma, \tau)^{-1} p \phi(\sigma, \tau, \tau)^{-1}\\
&\equal{(\ref{eq:A9.5},\ref{eq:A6.1})}&
\phi(\sigma, \tau, \sigma)p \phi(\rho, \tau, \tau)^{-1}
\equal{\equref{A9.4}} \phi(\sigma, \tau, \sigma)\phi(\rho, \rho, \tau);\\
\phi(\tau, \rho, \rho)\phi(\sigma, \sigma, \rho)&\equal{\equref{A6.1}}&
\varepsilon_\rho \phi(\sigma, \rho, \rho)\phi(\sigma, \sigma, \rho)
\equal{\equref{A9.4}}p\varepsilon_\rho ,
\end{eqnarray*}
and this finishes the proof.
\end{proof}

\begin{theorem}\thlabel{A10}
Let $\phi$ be a happy normalized 3-cocycle. $\phi$ is completely determined by
$\varepsilon_\sigma$, $\varepsilon_\tau$, $\varepsilon_\rho$,
$a=\phi(\tau, \sigma, \sigma)$ and $b=\phi(\sigma, \tau, \sigma)$. More precisely, we have
\begin{eqnarray}
a&=&\phi(\tau, \sigma, \sigma)=p\phi(\sigma, \tau, \tau)=p\phi(\tau, \tau, \sigma)^{-1}
=\phi(\sigma, \sigma, \tau)^{-1}\nonumber\\
&=&\varepsilon_\sigma \phi(\rho, \sigma, \sigma)=p\varepsilon_\sigma \phi(\sigma, \rho, \rho)=
\varepsilon_\sigma \phi(\sigma, \sigma, \rho)^{-1}=p\varepsilon_\sigma \phi(\rho, \rho, \sigma)^{-1}
\nonumber\\
&=&p\varepsilon_\tau \phi(\rho, \tau, \tau)=\varepsilon_\tau \phi(\tau, \rho, \rho)=
p\varepsilon_\tau \phi(\tau, \tau, \rho)^{-1}=\varepsilon_\tau \phi(\rho, \rho, \tau)^{-1}
\eqlabel{A10.1}\\
b&=&\phi(\sigma, \tau, \sigma)=p\varepsilon_\sigma \phi(\rho, \sigma, \rho)
=p\varepsilon_\tau \phi(\tau, \rho, \tau)\nonumber\\
&=&p\phi(\tau, \sigma, \tau)^{-1}=\varepsilon_\sigma \phi(\sigma, \rho, \sigma)^{-1}=
\varepsilon_\tau \phi(\rho, \tau, \rho)^{-1}\hspace{1mm}.\nonumber
\end{eqnarray}
\end{theorem}

\begin{proof}
Remark first that some of the cocycle conditions simplify if $\phi$ is happy. Using these
simplified cocycle relations, we compute
\begin{eqnarray*}
\phi(\tau, \tau, \sigma)^{-1}&\equal{\equref{A9.4}}&p \phi(\tau, \sigma, \sigma)=pa\hspace{1mm},\\
\phi(\sigma, \tau, \tau)&\equal{\equref{A9.6}}&\phi(\tau, \tau, \sigma)^{-1}\equal{\equref{A9.4}}p 
\phi(\tau, \sigma, \sigma)=p a\hspace{1mm},\\
\phi(\sigma, \sigma, \tau)^{-1}&\equal{\equref{A9.4}}&p\phi(\sigma, \tau, \tau)=a\hspace{1mm},\\
\phi(\rho, \sigma, \sigma)&\equal{\equref{A6.1}}&\varepsilon_\sigma \phi(\tau, \sigma, \sigma)
=a \varepsilon_\sigma\hspace{1mm},\\
\phi(\rho, \tau, \tau)&\equal{\equref{A6.1}}&\varepsilon_\tau \phi(\sigma, \tau, \tau)
=pa \varepsilon_\tau\hspace{1mm},\\
\phi(\sigma, \sigma,\rho)&\equal{\equref{A6.2}}&\varepsilon_\sigma \phi(\sigma, \sigma, \tau)
=a^{-1} \varepsilon_\sigma\hspace{1mm},\\
\phi(\tau, \tau, \rho)&\equal{\equref{A6.2}}&\varepsilon_\tau \phi(\tau, \tau, \sigma)
\equal{\equref{A9.4}}pa^{-1} \varepsilon_\tau\hspace{1mm},\\
\phi(\sigma, \rho, \rho)&\equal{\equref{A9.4}}&p\phi(\sigma, \sigma, \rho)^{-1} =pa 
\varepsilon_\sigma\hspace{1mm},\\
\phi(\tau, \rho, \rho)&\equal{\equref{A9.4}}&p\phi(\tau, \tau, \rho)^{-1} 
=a \varepsilon_\tau\hspace{1mm},\\
\phi(\rho, \rho, \sigma)&\equal{\equref{A9.4}}&p\phi(\rho, \sigma, \sigma)^{-1}=
pa^{-1} \varepsilon_\sigma\hspace{1mm},\\
\phi(\rho, \rho, \tau)&\equal{\equref{A9.4}}&p\phi(\rho, \tau, \tau)^{-1}
=a^{-1} \varepsilon_\tau \hspace{1mm},\\
\phi(\tau, \sigma, \tau)&\equal{\equref{A9.5}}&p\phi(\sigma, \tau, \sigma)^{-1}
=pb^{-1}\hspace{1mm},\\
\phi(\tau, \rho, \tau)&\equal{\equref{A6.3}}&\varepsilon_\tau \phi(\tau, \sigma, \tau)^{-1}=
p\varepsilon_\tau b\hspace{1mm},\\
\phi(\rho, \tau, \rho)&\equal{\equref{A9.5}}&p\phi(\tau, \rho, \tau)^{-1}=\varepsilon_\tau 
b^{-1}\hspace{1mm},\\
\phi(\rho, \sigma, \rho)&\equal{\equref{A6.3}}&\varepsilon_\rho \phi(\rho, \tau, \rho)^{-1}
=\varepsilon_\rho\varepsilon_\tau b=p\varepsilon_\sigma b\hspace{1mm},\\
\phi(\sigma, \rho, \sigma)&\equal{\equref{A6.3}}&\varepsilon_\sigma \phi(\sigma, \tau, \sigma)^{-1}
=\varepsilon_\sigma b^{-1}\hspace{1mm},
\end{eqnarray*} 
as we claimed, so our proof is complete.
\end{proof}

The maps $\phi$ described in \thref{A10} are indexed by the following parameters: 
$\varepsilon_\sigma$, $\varepsilon_\tau$, $\varepsilon_\rho\in \{- 1, 1\}$ and 
$a, b\in k^*$. It is a routine computation to verify that they all satisfy 
(\ref{eq:A6.1}-\ref{eq:A6.3}) and (\ref{eq:A9.4}-\ref{eq:A9.6}), hence they are all $3$-cocycles, 
by \prref{A9}. This tells us that there is a bijection from the subgroup 
$Z^3_h(C_2\times C_2, k^*)$ of $Z^3(C_2\times C_2, k^*)$ 
consisting of happy normalized cocycles, to the set $C_2^3\times (k^*)^2$.

Let $H_1$ be the subset of $Z^3_h(C_2\times C_2, k^*)$ for which the corresponding
parameters $a$ and $b$ are equal to $1$. This is also the subset of $Z^3_h(C_2\times C_2, k^*)$
consisting of cocycles $\phi$ for which $\phi(\tau, \sigma, \sigma)=\phi(\sigma, \tau, \sigma)=1$.
It is then clear that $H_1$ is a subgroup of $Z^3_h(C_2\times C_2, k^*)$, consisting of $8$ 
elements:
$$
H_1=\{\phi_X\mid X\subseteq \{\sigma,\tau,\rho\}\}.
$$
$\phi_\emptyset$ is the trivial $3$-cocycle and for a 
non-empty subset $X$ of $\{\sigma, \tau, \rho\}$, $\phi_X$ is the $3$-cocycle defined as follows: 
$\varepsilon_x=-1$ if and only if $x\in X$. The multiplication on $H_1$ is the following:
$$
\phi_X\phi_Y= \phi_{X\Delta Y},
$$
where $X\Delta Y=(X\setminus Y)\cup (Y\setminus X)$ is the symmetric difference of the sets $X$ 
and $Y$. $\phi_X$ is an even cocycle if and only if $|X|$ is even. It follows that 
$H_1\cong C_2\times C_2\times C_2$.

Since a normalized cocycle takes the value $1$ if one of the three entries is equal to $e$,
we can view them as functions $\{\sigma,\tau,\rho\}^3\to k^*$.
From the description in \thref{A10}, it follows that the eight cocycles in $H_1$ are invariant
under permutation: $\phi_X\circ s=\phi_X$, for all $X\subset\{\sigma,\tau,\rho\}$
and $s\in S_3$.

Let $\{P_e,P_\sigma,P_\tau,P_\rho\}$ be the basis of $k[G]^*\cong k^G$ dual to the basis
$\{e,\sigma,\tau,\rho\}$ of $k[G]$. Then the following elements of $k^{G\times G\times G}
\cong k[G]^*\ot k[G]^*\ot k[G]^*$ are invariant under permutations:
$$
X=\sum_{s\in S_3} P_{s(\sigma)}\ot P_{s(\tau)}\ot P_{s(\rho)},
$$
$$
X_x= P_x\ot P_x\ot P_x~~~~(x\in\{\sigma,\tau,\rho\}),
$$
$$
X_{x,y}= P_x\ot P_y\ot P_y +P_y\ot P_x\ot P_y + P_y\ot P_y\ot P_x
~~~~(x\neq y\in\{\sigma,\tau,\rho\}).
$$
Viewed as maps $G\times G\times G\to k^*$, these can also be described as follows:
$$
X(x,y,z)=
\begin{cases}
1&{\rm if}~ \{x,y,z\}=\{\sigma,\tau,\rho\}\\
0&{\rm otherwise}
\end{cases},
$$
$$
X_\sigma(x,y,z)=
\begin{cases}
1&{\rm if}~x=y=z=\sigma\\
0&{\rm otherwise}
\end{cases},
$$
and $X_{\sigma,\tau}(x,y,z)=1$ if one element of $(x,y,z)$ equals $\sigma$ and the two other
ones equal $\tau$, and $X_{\sigma,\tau}(x,y,z)=0$ otherwise.

Also observe that the $X$, $X_x$ and $X_{x,y}$ are orthogonal. From \thref{A10},
we now deduce the following formulas for the elements of $H_1$,
\begin{eqnarray}
\phi_{\{\sigma\}}&=&1-2(X_\sigma + X_{\sigma,\tau}+X_{\rho,\tau}+X_{\rho,\sigma}+X)\nonumber\\
\phi_{\{\tau\}}&=&1-2(X_\tau + X_{\sigma,\rho}+X_{\tau,\rho}+X_{\sigma,\tau}+X)\nonumber\\
\phi_{\{\rho\}}&=&1-2(X_\rho + X_{\sigma,\tau}+X_{\rho,\tau}+X_{\sigma,\rho}+X)\nonumber\\
\phi_{\{\sigma,\tau\}}&=&1-2(X_\sigma+X_\tau+X_{\rho,\tau}+X_{\tau,\rho}
+X_{\rho,\sigma}+X_{\sigma,\rho})\eqlabel{cocycleprim}\\
\phi_{\{\sigma,\rho\}}&=&1-2(X_\sigma+X_\rho+X_{\rho,\sigma}+X_{\sigma,\rho})\nonumber\\
\phi_{\{\tau,\rho\}}&=&1-2(X_\tau+X_\rho+X_{\rho,\tau}+X_{\tau,\rho})\nonumber\\
\phi_{\{\sigma,\tau,\rho\}}&=&1-2(X_\sigma+X_\tau+X_\rho+X_{\rho,\sigma}+
X_{\sigma,\tau}+ X_{\tau,\rho}+X).\nonumber
\end{eqnarray}
For any $b\in k^*$, let $g_b$ be the cocycle that we obtain taking $a=1$, 
$\varepsilon_\sigma=\varepsilon_\tau=\varepsilon_\rho=1$ in \thref{A10}. We have
\begin{equation}\eqlabel{cocyclesec}
g_b(x, y, z)=\left\{
\begin{array}{cl}
b&\mbox{if $(x, y, z)\in\{(\sigma, \tau, \sigma), (\rho, \sigma, \rho), (\tau, \rho, \tau)\}$}\\
b^{-1}&\mbox{if $(x, y, z)\in\{(\tau, \sigma, \tau), (\sigma, \rho, \sigma), (\rho, \tau, \rho)\}$}\\
1&\mbox{otherwise}
\end{array}\right. .
\end{equation}
For any $a\in k^*$, let $h_a$ be the cocycle that we obtain taking $b=1$, 
$\varepsilon_\sigma=\varepsilon_\tau=\varepsilon_\rho=1$ in \thref{A10}. Thus
\[
h_a(x, y, z)=\left\{
\begin{array}{cl}
a&\mbox{if $e\not=x\not=y=z\not=e$}\\
a^{-1}&\mbox{if $e\not=x=y\not=z\not=e$}\\
1&\mbox{otherwise}
\end{array}\right. .
\]

It is clear that $H_2=\{g_b\mid b\in k^*\}$ and $H_3=\{h_a\mid a\in k^*\}$ are subgroups of
$Z^3_h(C_2\times C_2, k^*)$. Therefore $Z^3_h(C_2\times C_2, k^*)=H_1\times H_2\times H_3$,
and we have the following result.

\begin{corollary}\colabel{A11}
We have an isomorphism of abelian groups
$$Z^3_h(C_2\times C_2, k^*)\cong C_2\times C_2\times C_2 \times k^*\times k^*.$$
\end{corollary}

It is easy to see that $h_ag_b$ is invariant under permutation if and only if $a=b=-1$.
Observe also that
\begin{eqnarray*}
h_{-1}g_{-1}&=&1-2(X_{\rho,\sigma}+X_{\sigma,\rho}+X_{\sigma,\tau}+X_{\tau,\sigma}
+X_{\tau,\rho}+X_{\rho,\tau});\\
h_{-1}g_{-1}\phi_{\{\sigma, \tau\}}&=&
1-2(X_\sigma+X_\tau+X_{\tau,\rho}+X_{\rho,\tau}).
\end{eqnarray*}
The subgroup $\tilde{H}$ of $Z^3_h(C_2\times C_2, k^*)$ consisting of
cocycles invariant under permutation is the subgroup generated by $H_1$
and $h_{-1}g_{-1}$, and it follows from \coref{A11} that
$\tilde{H}\cong C_2^4$.

In order to compute $H^3(C_2\times C_2, k^*)$, we now have to figure out which happy normalized cocycles are coboundaries. 

$H^3(C_2\times C_2, k^*)$ is an epimorphic image of $Z^3_h(C_2\times C_2, k^*)$.
We have to figure out which happy normalized cocycles are coboundaries. 

\begin{proposition}\prlabel{A12}
If $X\neq \emptyset$, then $\phi_X$ is not a coboundary.
$h_a$ is a coboundary, for every $a\in k^*$. $g_b$ is a coboundary if and only if 
$b$ has a squareroot in $k^*$. 
\end{proposition}
\begin{proof}
The first statement follows from the fact that $\phi(x, x, x)=1$ for every $x$ if $\phi$ is a 
normalized coboundary, cf. \equref{A5.2a}.

$h_a$ can be written as a coboundary in two different ways: take 
$g:\ C_2\times C_2\to k^*$ as in \equref{A5.1}, 
with $b_i=1$, for $i=1,\cdots, 6$, $a_1=a_2=a_3=a$ and $c=1$, or $c=a$, and all the
$a_i$ and $b_i$ equal to $1$. It follows from \equref{A5.2} 
that $\Delta_2(g)$ is happy, and that $\Delta_2(g)(\tau,\sigma,\sigma)=a$ and 
$\Delta_2(g)(\sigma,\tau,\sigma)=1$. Applying \thref{A10}, we see that $\Delta_2(g)=h_a$.

Assume that $b=d^2$, and consider $g:\ C_2\times C_2\to k^*$ as in \equref{A5.1},  
now with $a_1=a_2=a_3=b_4=b_5=b_6=d$, $b_1=b_2=b_3=1$, $c=1$. It follows from 
\equref{A5.2} that $\Delta_2(g)$ is happy, and that 
$\Delta_2(g)(\tau,\sigma,\sigma)=1$ and 
$\Delta_2(g)(\sigma,\tau,\sigma)=d^2=b$. \thref{A10} tells us that $\Delta_2(g)=g_b$, 
so $g_b$ is coboundary. 

Conversely, if $g_b$ is coboundary then  $g_b=\Delta_2(g)$, with
$g:\ C_2\times C_2\to k^*$ given by \equref{A5.1}, for some $a_1, a_2, a_3, 
b_1, b_2, b_3, b_4, b_5, b_6, c\in k^*$. From \equref{A5.2} and the description of 
$g_b$, it follows that
\[
1=g_b(\sigma,\sigma,\tau)=b_1b_5a_1^{-1}c^{-1}~~;~~
1=g_b(\tau,\sigma,\rho)=b_4^{-1}b_5a_2a_3^{-1};
\]
\[
1=g_b(\rho,\rho,\sigma)=b_3b_6a_3^{-1}c^{-1}~~;~~
1=g_b(\sigma,\rho,\tau)=b_5^{-1}b_6a_1a_2^{-1}.
\]
From the first two formulas it follows that $b_1=a_1a_2a_3^{-1}b_4^{-1}c$, and
combining the other two formulas we obtain that $b_3=a_1a_2^{-1}a_3b_5^{-1}c$.
Using \equref{A5.2}, we now find that
\begin{eqnarray*}
b&=&g_b(\sigma,\sigma,\tau)=b_1^{-1}b_3^{-1}b_4b_5\\
&=&a_1^{-1}a_2^{-1}a_3b_4c^{-1} a_1^{-1}a_2a_3^{-1}b_5c^{-1}b_4b_5
= (a_1^{-1}b_4b_5c^{-1})^2
\end{eqnarray*}
is a square in $k^*$.
\end{proof}
 
\begin{corollary}\colabel{A13}
$H^3(C_2\times C_2, k^*)=C_2\times C_2\times C_2\times k^*/k^{*(2)}$, where 
$k^{*(2)}=\{\a^2\mid \a\in k^*\}$.
\end{corollary}

\begin{corollary}\colabel{A14}
If every element of $k$ has a squareroot (for example, if $k$ is algebraically closed), 
then $H^3(C_2\times C_2,k^*)=C_2\times C_2\times C_2$.
\end{corollary}
  
So a non-strict monoidal structure on ${\rm Vect}^{C_2\times C_2}$ is defined by the 
one of the $3$-cocycles defined in \equref{cocycleprim}, or by a $3$-cocycle as in 
\equref{cocyclesec} with $b\in k^*\backslash k^{*(2)}$. All the remaining monoidal 
structures of ${\rm Vect}^{C_2\times C_2}$ are monoidal isomorphic to the strict 
monoidal structure of ${\rm Vect}^{C_2\times C_2}$.   

All our computations are over fields of characteristic different from 2; they 
can be extended easily to the case where ${\rm char}(k)=2$. Then all $\varepsilon_x=1$, 
and we obtain the following result.
 
\begin{proposition}\prlabel{A15}
Let $k$ be a field of characteristic 2. Then $H^3(C_2\times C_2, k^*)=k^*/k^{*(2)}$. 
If every element of $k$ has a squareroot, then $H^3(C_2\times C_2,k^*)=1$.
\end{proposition}

\begin{remark}\relabel{A16}
Let $C_2=\{e,\alpha\}$. If the characteristic of $k$ is different from $2$, then
$Z^3(C_2,k^*)$ contains two cocycles. The nontrivial cocycle $\phi$ takes the
value $-1$ at $(\alpha,\alpha,\alpha)$, and $1$ elsewhere. Otherwise stated
$$\phi=1-2P_\alpha\ot P_\alpha\ot P_\alpha,$$
if $\{P_e,P_\alpha\}$ is the basis of $k[C_2]^*$ dual to $\{e,\alpha\}$.

Now we have three Hopf algebra morphisms $t_1,t_2,t_3:\
k[C_2]^*\to k[C_2\times C_2]^*$. These are given by the formulas
$$\begin{array}{ccc}
t_1(P_e)=P_e+P_\sigma&&t_1(P_\alpha)=P_\tau+P_\rho\\
t_2(P_e)=P_e+P_\tau&&t_2(P_\alpha)=P_\sigma+P_\rho\\
t_3(P_e)=P_e+P_\rho&&t_3(P_\alpha)=P_\tau+P_\sigma
\end{array}
$$
The $t_i$ induce group morphisms $t_i:\ Z^3(C_2,k^*)\to Z^3(C_2\times C_2,k^*)$.
Now we easily see that
\begin{eqnarray*}
t_1(\phi)&=& 1-2(P_\tau+P_\rho)\ot(P_\tau+P_\rho)\ot (P_\tau+P_\rho)=\phi_{\{\tau,\rho\}}\\
t_2(\phi)&=& 1-2(P_\sigma+P_\rho)\ot(P_\sigma+P_\rho)\ot (P_\sigma+P_\rho)=\phi_{\{\sigma,\rho\}}\\
t_3(\phi)&=& 1-2(P_\tau+P_\sigma)\ot(P_\tau+P_\sigma)\ot (P_\tau+P_\sigma)=
h_{-1}g_{-1}\phi_{\{\sigma,\tau\}}
\end{eqnarray*}
If $k$ contains a squareroot $i$ of $-1$, then $h_{-1}g_{-1}$ is a coboundary,
and $[t_1(\phi)][t_2(\phi)]=[t_3(\phi)]$ in $H^3(C_2\times C_2,k^*)$.
\end{remark}

\section{Computation of the abelian cocycles on the Vierergruppe}\selabel{braided} 
\setcounter{equation}{0}
\subsection{Computation of the quadratic forms}\selabel{4.1}
Throughout this Section, we assume that ${\rm char}(k)\neq 2$.
In order to describe the braidings of ${\rm Vect}^{C_2\times C_2}$, we have to compute 
$H^3_{\rm ab}(C_2\times C_2, k^*)$, see \seref{1.3}. To this end,
we will make use of the Eilenberg-Mac Lane Theorem, see \seref{1.5}: we will compute
$QF(C_2\times C_2, k^*)$.

\begin{lemma}\lelabel{4.1}
$Q:\ C_2\times C_2\to k^*$ is a quadratic form if and only if
\begin{enumerate}
\item[a)] $Q(e)=1$;
\item[b)] $Q(\sigma)^4=Q(\tau)^4=Q(\rho)^4=1$;
\item[c)] $Q(\sigma)^2Q(\tau)^2Q(\rho)^2=1$.
\end{enumerate}
\end{lemma}

\begin{proof}
Assume first that $Q$ is a quadratic form. a) follows after we take $x=y=z=e$ in
\equref{quadraticform}. b) Take $x=y=z=\sigma$ in \equref{quadraticform}.
Since $\sigma^3=\sigma$ and $\sigma^2=e$, it follows that $Q(\sigma)^4=1$.
c) Take $x=y=\sigma$ and $\rho=\tau$ in \equref{quadraticform}. Then we find
that $Q(\sigma)^2Q(\tau)^2=Q(\rho)^2$. Multiplying both sides by $Q(\rho)^2$,
we find c).

Conversely, assume that $Q$ satisfies conditions a), b) and c). 
Then $Q(x^{-1})=Q(x)$ is automatically satisfied, since $x=x^{-1}$, for all $x\in 
C_2\times C_2$. To prove \equref{quadraticform}, we distinguish several cases.

Case 1: $e\in \{x,y,z\}$, say $x=e$. Then \equref{quadraticform} reduces to 
$Q(yz)Q(y)Q(z)=Q(y)Q(z)Q(yz)$, which is satisfied.

Case 2: $e\not\in \{x,y,z\}$.

2a) $|\{x,y,z\}|=1$: $x=y=z$.  \equref{quadraticform} reduces to
$Q(x)^4=Q(e)^3=1$.

2b) $|\{x,y,z\}|=2$, say $x=y=\sigma$, $z=\tau$.  \equref{quadraticform} reduces to
$Q(\tau)Q(\sigma)^2Q(\tau)=Q(e)Q(\rho)Q(\rho)$.

2c) $|\{x,y,z\}|=3$, say $x=\sigma$, $y=\tau$, $z=\rho$. 
\equref{quadraticform} reduces to $Q(e)Q(\sigma)Q(\tau)Q(\rho)=
Q(\rho)Q(\tau)Q(\sigma)$.
\end{proof}

Assume first that $k$ contains $i$, a primitive fourth root of 1. Then $Q$ is a quadratic form if and only if
$Q(e)=1$, $Q(\sigma),Q(\tau),Q(\rho)\in \{\pm 1, \pm i\}$ and 
$Q(\sigma)Q(\tau)Q(\rho)=\pm1$. Then $QF(C_2\times C_2, k^*)$ has $32$ elements, summarized in 
\taref{QF}.

\begin{table}[htb]
\begin{center}
\begin{tabular}{c|r|r|r|r|r|r|r|r|}
 & $I$ & $A$ & $B$ & $C$ & $AB$ & $AC$ & $BC$ & $ABC$ \\
\hline
$Q(\sigma)$& $1$& $1$ & $1$& $-1$ & $1$ & $-1$ & $-1$ & $-1$ \\
\hline
$Q(\tau)$ & $1$& $1$ & $-1$& $1$ & $-1$ & $1$ & $-1$ & $-1$ \\
\hline
$Q(\rho)$& $1$& $-1$ & $1$& $1$ & $-1$ & $-1$ & $1$ & $-1$ \\
\hline
\end{tabular}\\

\vspace*{5mm}

\begin{tabular}{c|r|r|r|r|r|r|r|r|}
 & $E_1$ & $AE_1$ & $BE_1$ & $CE_1$ & 
 $ABE_1$ & $ACE_1$ & $BCE_1$ & $ABCE_1$ \\
\hline
$Q(\sigma)$ & $i$ & $i$ & $i$ & $-i$ & $i$ & $-i$ & $-i$ & $-i$\\
\hline
$Q(\tau)$ & $i$ & $i$ & $-i$ & $i$ & $-i$ & $i$ & $-i$ & $-i$\\
\hline
$Q(\rho)$ & $1$ & $-1$ & $1$ & $1$ & $-1$ & $-1$ & $1$ & $-1$\\
\hline 
\end{tabular}\\

\vspace*{5mm}

\begin{tabular}{c|r|r|r|r|r|r|r|r|}
 & $E_2$ & $AE_2$ & $BE_2$ & $CE_2$ & 
 $ABE_2$ & $ACE_2$ & $BCE_2$ & $ABCE_2$ \\
\hline
$Q(\sigma)$ & $i$ & $i$ & $i$ & $-i$ & $i$ & $-i$ & $i$ & $-i$\\
\hline
$Q(\tau)$ & $1$ & $1$ & $-1$ & $1$ & $-1$ & $1$ & $-1$ & $-1$\\
\hline
$Q(\rho)$ & $i$ & $-i$ & $i$ & $i$ & $-i$ & $-i$ & $i$ & $-i$\\
\hline 
\end{tabular}\\

\vspace*{5mm}

\begin{tabular}{c|r|r|r|r|r|r|r|r|}
 & $E_3$ & $AE_3$ & $BE_3$ & $CE_3$ & 
 $ABE_3$ & $ACE_3$ & $BCE_3$ & $ABCE_3$ \\
\hline
$Q(\sigma)$ & $1$ & $1$ & $1$ & $-1$ & $1$ & $-1$ & $-1$ & $-1$\\
\hline
$Q(\tau)$ & $i$ & $i$ & $-i$ & $i$ & $-i$ & $i$ & $-i$ & $-i$\\
\hline
$Q(\rho)$ & $i$ & $-i$ & $i$ & $i$ & $-i$ & $-i$ & $i$ & $-i$\\
\hline 
\end{tabular}
\end{center}

\vspace*{3mm}

\caption{The $32$ elements of $QF(C_2\times C_2, k^*)$\talabel{QF}}
\end{table}
Thus $QF(C_2\times C_2, k^*)$ is the abelian group consisting of $I$, $A$, $B$, $C$, $AB$, 
$AC$, $BC$, $ABC$; 
$E_j, AE_j, BE_j$, $CE_j, ABE_j, ACE_j, BCE_j, ABCE_j$, $j=1,2,3$, 
with relations
\begin{equation}\eqlabel{relabcocK}
\begin{array}{c}
A^2=B^2=C^2=I,~~E_1^2=BC,~~E_2^2=AC,~~E_3^2=AB\\
E_1E_2=CE_3,~~E_1E_3=BE_2,~~E_2E_3=AE_1,
\end{array}
\end{equation}
for all $j={1, 2,3}$. Hence $QF(C_2\times C_2, k^*)\cong C_4\times C_4\times C_2$, 
because it is an abelian group of order $32$ that contains precisely seven elements 
of order two and all the other non-trivial elements have order four.

If $k$ does not contain a fourth root of $1$, then we clearly have
$$QF(C_2\times C_2, k^*)=\{I, A, B, C, AB, AC, BC, ABC\}\cong C_2\times C_2\times C_2.$$
This describes $QF(C_2\times C_2, k^*)\cong H^3_{\rm ab}(C_2\times C_2, k^*)$. Our
aim is now to compute explicitly the abelian cocycles corresponding to
the 32 quadratic forms.

\subsection{Computation of the abelian cocycles}\selabel{4.2}
By abuse of language, we will say that a 3-cocycle $\phi\in Z^3(C_2\times C_2,k^*)$
is abelian if it is the underlying cocycle of an abelian coycle
$ (\phi,\Rr)\in Z^3_{\rm ab}(C_2\times C_2,k^*)$, or, equivalently, if
$\pi^{-1}(\phi)\neq\emptyset$, where $\pi:\ H^3_{\rm ab}(C_2\times C_2,k^*)\to
H^3(C_2\times C_2,k^*)$ is induced by the projection on the first component.

\begin{lemma}\lelabel{4.2}
For $b\in k^*$, $g_b$ is an abelian $3$-cocycle if and only if 
$g_b$ is coboundary as a $3$-cocycle.  
\end{lemma}

\begin{proof}
Assume that $g_b$ is an abelian $3$-cocycle via a certain map 
${\cal R}: G\times G\ra k^*$. 
Taking $x=y=\sigma$ and $z=\tau$ in \equref{r1coc} we find that
$$
{\cal R}(e, \tau)g_b(\sigma, \tau, \sigma)=g_b(\sigma, \sigma, \tau){\cal R}(\sigma, \tau)^2
g_b(\tau, \sigma, \sigma).$$
By \equref{cocyclesec} we obtain $b={\cal R}(\sigma, \tau)^2$, and so $g_b$ is a coboundary 
$3$-cocycle on $C_2\times C_2$, cf. \prref{A12}.   
\end{proof}

Our next aim is to compute $\pi^{-1}(\phi_\emptyset)$. This allows to compute
$\pi^{-1}(\phi)$, for every coboundary $\phi$.

\begin{proposition}\prlabel{ab3cocycfromtriv} 
The subgroup $\pi^{-1}(\phi_\emptyset)$ of $H^3_{\rm ab}(C_2\times C_2,k^*)$ is
isomorphic to $C_2\times C_2\times C_2$. Its elements are of the form
$[(1,\Rr)]$, with $\Rr$ given by the following data:
\begin{eqnarray*}
&&{\cal R}(x, x)=\mu_x~~{\rm with}~~\mu_x^2=1,~~\forall~~x\in\{\sigma, \tau, \rho\},~~\\
&&{\cal R}(\sigma, \tau)=1,~~
{\cal R}(\tau, \sigma)=\mu_\sigma\mu_\tau\mu_\rho,~~
{\cal R}(\sigma, \rho)=\mu_\sigma,~~\\
&&{\cal R}(\rho, \sigma)=\mu_\tau\mu_\rho,~~
{\cal R}(\tau, \rho)=\mu_\sigma\mu_\rho,~~{\cal R}(\rho, \tau)=\mu_\tau,
\end{eqnarray*}
Moreover $EM(\pi^{-1}(\phi))$ is the subgroup of $QF(C_2\times C_2,k^*)$
generated by $A$, $B$ and $C$.
\end{proposition}

\begin{proof}
Let $(1,\Rr)$ be an abelian cocycle.
The relations (\ref{eq:r1coc}-\ref{eq:r2coc}) reduce to
\begin{equation}\eqlabel{rcocsimpl}
{\cal R}(xy, z)={\cal R}(x, z){\cal R}(y, z)~~{\rm and}~~
{\cal R}(x, yz)={\cal R}(x, y){\cal R}(x, z),
\end{equation}
for all $x,y,z\in C_2\times C_2$. This means that $\Rr$ is a bilinear map.
Observe that ${\rm Bil}(C_2^2\times C_2^2, k^*)\cong \Hom(C_2^2,(C_2^2)^*)
\cong {\rm End}(C_2^2)= C_2^4$. Here $(C_2^2)^*$ is the character group of
$C_2^2$, and the characters take values in $\{1,-1\}$. It follows that $\Rr$ takes values in
$\{1,-1\}$, and there are 16 different maps for $\Rr$. These are completely determined
by the values $\Rr(\sigma,\sigma)=\mu_\sigma,\Rr(\tau,\tau)=\mu_\tau,
\Rr(\rho,\rho)=\mu_\rho,\Rr(\sigma,\tau)=\alpha\in \{1,-1\}$. Indeed, the other values
of $\Rr$ follow from the bilinearity of $\Rr$:
\begin{eqnarray*}
\Rr(\sigma,\rho)&=&\Rr(\sigma,\sigma)\Rr(\sigma,\tau)=\alpha\mu_\sigma\\
\Rr(\rho,\tau)&=&\Rr(\sigma,\tau)\Rr(\tau,\tau)=\alpha\mu_\tau\\
\Rr(\rho,\sigma)&=&\Rr(\rho,\tau)\Rr(\rho,\rho)=\alpha\mu_\tau\mu_\rho\\
\Rr(\tau,\sigma)&=&\Rr(\sigma,\sigma)\Rr(\rho,\sigma)=\alpha\mu_\sigma\mu_\tau\mu_\rho\\
\Rr(\tau,\rho)&=&\Rr(\sigma,\rho)\Rr(\rho,\rho)=\alpha\mu_\sigma\mu_\rho .
\end{eqnarray*}
Now let $\Rr$ and $\Rr_-$ be two bilinear forms that reach the same values at
$(\sigma,\sigma)$, $(\tau,\tau)$ and $(\rho,\rho)$, and assume that
$\Rr(\sigma,\tau)=1=-\Rr_-(\sigma,\tau)$.Then  ${\cal R}$ and ${\cal R}_{-}$ 
are cohomologous as abelian $3$-cocycles on $C_2\times C_2$. To see this take 
$g:\ C_2\times C_2\ra k^*$ as in \equref{A5.1} with $a_1=a_2=a_3=b_4=b_5=b_6=-1$ and 
$b_1=b_2=b_3=c=1$. Then one can easily verify that $\Delta_2(g)=1$ and 
${\cal R}(x, y)=g(x, y)^{-1}g(y, x){\cal R}_{-}(x, y)$, for all $x, y\in C_2\times C_2$.

It is easy to see that the images under $EM$ of the eight bilinear forms that take the
value $1$ at $(\sigma,\tau)$ are $I$, $A$, $B$, $C$, $AB$, $AC$, $BC$ and $ABC$.
It then follows from the Eilenberg-Mac Lane \thref{EM} that these eight bilinear forms
represent different cohomology classes in  $H^3_{\rm ab}(C_2\times C_2,k^*)$.
\end{proof}

In \taref{trivial}, we present the eight abelian cocycles representing the elements of
$\pi^{-1}(\phi)$. Remark that the braidings associated to $I$, $AB$, $AC$ and $BC$
are symmetries on ${\rm Vect}^{C_2\times C_2}$. 
Indeed, by \equref{rcocsimpl} it follows that 
${\cal R}^{-1}(x, y)={\cal R}(x^{-1}, y)={\cal R}(x, y)$, for all $x, y\in C_2\times C_2$, so 
${\rm Vect}^{C_2\times C_2}$ is symmetric monoidal if and only if ${\cal R}(x, y)={\cal R}(y, x)$, 
for all $x, y\in C_2\times C_2$. Now \taref{trivial} below shows that only $I$, $AB$, $AC$ 
and $BC$ satisfy this condition.

\begin{table}[htb]
\begin{center}
\begin{tabular}{c|r|r|r|r|r|r|r|r|}
$EM(1,\Rr)$ & $I$ & $A$ & $B$ & $C$ & $AB$ & $AC$ & $BC$ & $ABC$ \\
\hline
${\cal R}(\sigma, \sigma)$& $1$& $1$ & $1$& $-1$ & $1$ & $-1$ & $-1$ & $-1$ \\
\hline
${\cal R}(\tau, \tau)$ & $1$& $1$ & $-1$& $1$ & $-1$ & $1$ & $-1$ & $-1$ \\
\hline
${\cal R}(\rho, \rho)$& $1$& $-1$ & $1$& $1$ & $-1$ & $-1$ & $1$ & $-1$ \\
\hline
${\cal R}(\sigma, \tau)$& $1$& $1$ & $1$& $1$ & $1$ & $1$ & $1$ & $1$ \\
\hline
${\cal R}(\tau, \sigma)$& $1$& $-1$ & $-1$& $-1$ & $1$ & $1$ & $1$ & $-1$ \\
\hline
${\cal R}(\sigma, \rho)$& $1$& $1$ & $1$& $-1$ & $1$ & $-1$ & $-1$ & $-1$ \\
\hline
${\cal R}(\rho, \sigma)$& $1$& $-1$ & $-1$& $1$ & $1$ & $-1$ & $-1$ & $1$ \\
\hline
${\cal R}(\tau, \rho)$& $1$& $-1$ & $1$& $-1$ & $-1$ & $1$ & $-1$ & $1$ \\
\hline
${\cal R}(\rho, \tau)$& $1$& $1$ & $-1$& $1$ & $-1$ & $1$ & $-1$ & $-1$ \\
\hline
\mbox{otherwise}& $1$& $1$ & $1$& $1$ & $1$ & $1$ & $1$ & $1$ \\
\hline 
\end{tabular}
\end{center}
\vspace*{3mm}
\caption{The eight abelian $3$-cocycles with trivial underlying $3$-cocycle
\talabel{trivial}}
\end{table}

We still have to compute the abelian $3$-cocycles in $\pi^{-1}(\phi_X)$, with $X$ a non-empty subset of 
$\{\sigma, \tau, \rho\}$. 

\begin{proposition}\prlabel{ab3cocycfromnontriv}
Let $X\subseteq \{\sigma, \tau, \rho\}$ be a non-empty subset. Then $\phi_X$ is an abelian 
$3$-cocycle if and only if it is even and $k$ contains a primitive fourth root of $1$. 
\end{proposition} 

\begin{proof}
Assume $(\phi_X, {\cal R})$ abelian for a certain ${\cal R}$ and denote $\mu_x:={\cal R}(x, x)$, 
for all $x\in \{\sigma, \tau, \rho\}$. Taking $x=y=z$ in \equref{r1coc} or \equref{r2coc} 
we get $\mu_x^2=\va_x$, for all $x\in \{\sigma, \tau, \rho\}$. If we take $x=z\not=y$ in 
\equref{r1coc} we obtain ${\cal R}(xy, x)=\phi_X(x, y, x)\mu_x{\cal R}(y, x)$, for all 
$x\not=y$ from $\{\sigma, \tau, \rho\}$. Thus, according to \equref{A10.1} we have   
\begin{equation}\eqlabel{phiXab1}
{\cal R}(\rho, \sigma)=\mu_\sigma{\cal R}(\tau, \sigma),~~
{\cal R}(\tau, \rho)=\mu_\rho\va_\rho\va_\tau{\cal R}(\sigma, \rho),~~
{\cal R}(\sigma, \tau)=\mu_\tau\va_\sigma\va_\rho{\cal R}(\rho, \tau). 
\end{equation}  
The same relations we get if we take $y=z\not=x$ in \equref{r1coc}. For $x=y\not=z$ in 
\equref{r1coc} we obtain 
$\phi_X(x, z, x)=\phi_X(x, x, z){\cal R}(x, z)^2\phi_X(z, x, x)$, 
and therefore, by \equref{A10.1} we deduce that 
\begin{equation}\eqlabel{phiXab2}
\begin{matrix}
{\cal R}(\sigma, \tau)^2=1,~~{\cal R}(\sigma, \rho)^2=\va_\sigma,~~
{\cal R}(\tau, \sigma)^2=p,\cr
{\cal R}(\tau, \rho)^2=\va_\sigma\va_\rho,~~
{\cal R}(\rho, \sigma)^2=\va_\rho\va_\tau,~{\cal R}(\rho, \tau)^2=\va_\tau.\cr
\end{matrix}
\end{equation}
Moving to \equref{r2coc}, for $z=x\not=y$ we obtain $\phi_X(x, y, x){\cal R}(x, xy)=
\mu_x{\cal R}(x, y)$, so 
\begin{equation}\eqlabel{phiXab3}
\begin{matrix}
{\cal R}(\sigma, \rho)=\mu_\sigma{\cal R}(\sigma, \tau),~~
\va_\sigma{\cal R}(\sigma, \tau)={\cal R}(\sigma, \rho)\mu_\sigma,~~
p{\cal R}(\tau, \rho)={\cal R}(\tau, \sigma)\mu_\tau,\cr
p\va_\tau{\cal R}(\tau, \sigma)={\cal R}(\tau, \rho)\mu_\tau,~~
p\va_\sigma{\cal R}(\rho, \tau)={\cal R}(\rho, \sigma)\mu_\rho,~~
\va_\tau{\cal R}(\rho, \sigma)={\cal R}(\rho, \tau)\mu_\rho.
\end{matrix}
\end{equation}
The same relations are obtained if we take $x=y\not=z$ among $\{\sigma, \tau, \rho\}$ in 
\equref{r2coc}, while for $y=z\not=x$ we obtain
$\phi_X(x, y, y)\phi_X(y, y, x)={\cal R}(x, y)^2\phi_X(y, x, y)$. This yields to 
\begin{equation}\eqlabel{phiXab4}
\begin{matrix}
{\cal R}(\sigma, \tau)^2=p,~~{\cal R}(\sigma, \rho)^2=\va_\tau\va_\rho,~~
{\cal R}(\tau, \sigma)^2=1,\cr
{\cal R}(\tau, \rho)^2=\va_\tau,~~
{\cal R}(\rho, \sigma)^2=\va_\sigma,~~
{\cal R}(\rho, \tau)^2=\va_\sigma\va_\rho.\cr
\end{matrix}
\end{equation}   
From the first equalities in \equref{phiXab2} and \equref{phiXab4} it follows that $p=1$, and so 
$\phi_X$ is necessarily even. Also, there exists $x\in\{\sigma, \tau, \rho\}$ such 
that $\va_x=-1$, so the equation $\mu_x^2=-1$ has a solution in $k$, and
$k$ contains a primitive fourth root of $1$. 

Conversely, if $\phi_X$ is even and $k$ contains a primitive fourth root of $1$, 
then $\phi_X$ is completely determined 
by $\mu_x$, $x\in\{\sigma, \tau, \rho\}$, and $\a:={\cal R}(\sigma, \tau)\in\{\pm 1\}$, cf. 
the first relation in \equref{phiXab4}. Actually, combining the relations in 
\equref{phiXab1}-\equref{phiXab4} we must have 
\begin{equation}\eqlabel{phiXabfin}
\begin{matrix}
{\cal R}(x, x)=\mu_x~~{\rm with}~~\mu_x^2=\va_x,~~\forall~~x\in\{\sigma, \tau, \rho\}~~;~~\cr
{\cal R}(\sigma, \tau)=\a\in\{\pm 1\}~~;~~
{\cal R}(\tau, \sigma)=\a\va_\rho\mu_\sigma\mu_\tau\mu_\rho~~;~~
{\cal R}(\sigma, \rho)=\a\mu_\sigma~~;~~\cr
{\cal R}(\rho, \sigma)=\a\va_\tau\mu_\tau\mu_\rho~~;~~
{\cal R}(\tau, \rho)=\a\va_\sigma\mu_\sigma\mu_\rho~~;~~
{\cal R}(\rho, \tau)=\a\mu_\tau.  \cr
\end{matrix}
\end{equation}   
Likewise, if ${\cal R}$ is defined by \equref{phiXabfin} then it can be easily checked 
that all the relations in (\ref{eq:phiXab1}-\ref{eq:phiXab4}) are satisfied, and so 
\equref{r1coc} and \equref{r2coc} are verified. 
\end{proof}

\begin{remark}
Consider an abelian $3$-cocycle $(\phi_X, {\cal R})$, and assume that the
corresponding braided monoidal structure is symmetric. If $X\neq \emptyset$,
then there exists $x\in X$, such that $\varepsilon_x=-1$ and $\Rr(x,x)=\mu_x=\pm i$,
by \equref{phiXabfin}. Then $\Rr(x,x)^{-1}=-\Rr(x,x)$, contradicting that the monoidal
structure is symmetric. We conclude that $X=\emptyset$, and ${\rm Vect}^{C_2\times C_2}$ 
admits only four types of symmetric monoidal structures,
namely the ones  corresponding to the 
abelian $3$-cocycles $I$, $AB$, $AC$ and $BC$, see \taref{trivial} for the description of
these cocycles.
\end{remark}

The abelian $3$-cocycles in \prref{ab3cocycfromnontriv} with $\a=1$ represent 
the same elements in $H^3_{\rm ab}(C_2\times C_2, k^*)$ as the ones with $\a=-1$. 
For this take $g$ as in the proof of \prref{ab3cocycfromtriv} to show that 
they are cohomologous as abelian $3$-cocycles. The abelian $3$-cocycles obtained from 
$\a=1$ are not cohomologous because of the theorem of Eilenberg and Mac Lane.

If $k$ does not contain a primitive fourth root of unity, then $\pi^{-1}(\Phi_X)=\emptyset$
for $X\neq\emptyset$, and $H^3_{\rm ab}(C_2\times C_2, k^*)\cong
C_2\times C_2\times C_2$, as described in \prref{ab3cocycfromtriv}.
Now assume that $k$ contains a primitive fourth root of unity $i$. The inverse images under
$\pi$ of $\phi_X$, $X=\{\sigma,\tau\},\{\sigma,\rho\},\{\tau,\rho\}$, each contain eight elements.
Their explicit description is given in \prref{ab3cocycfromnontriv}, and is summarized in
Tables \ref{ta:sigmatau}, \ref{ta:sigmarho} and \ref{ta:taurho}.

\begin{table}[htb]
\begin{center}
\begin{tabular}{c|r|r|r|r|r|r|r|r|}
$EM(\phi_{\{\sigma,\tau\}},\Rr)$
 & $E_1$ & $AE_1$ & $BE_1$ & $CE_1$ & 
 $ABE_1$ & $ACE_1$ & $BCE_1$ & $ABCE_1$ \\
\hline
${\cal R}(\sigma, \sigma)$ & $i$ & $i$ & $i$ & $-i$ & $i$ & $-i$ & $-i$ & $-i$\\
\hline
${\cal R}(\tau, \tau)$ & $i$ & $i$ & $-i$ & $i$ & $-i$ & $i$ & $-i$ & $-i$\\
\hline
${\cal R}(\rho, \rho)$ & $1$ & $-1$ & $1$ & $1$ & $-1$ & $-1$ & $1$ & $-1$\\
\hline
${\cal R}(\sigma, \tau)$ & $1$ & $1$ & $1$ & $1$ & $1$ & $1$ & $1$ & $1$\\
\hline 
${\cal R}(\tau, \sigma)$ & $-1$ & $1$ & $1$ & $1$ & $-1$ & $-1$ & $-1$ & $1$\\
\hline
${\cal R}(\sigma, \rho)$ & $i$ & $i$ & $i$ & $-i$ & $i$ & $-i$ & $-i$ & $-i$\\
\hline
${\cal R}(\rho, \sigma)$ & $-i$ & $i$ & $i$ & $-i$ & $-i$ & $i$ & $i$ & $-i$\\
\hline
${\cal R}(\tau, \rho)$ & $-i$ & $i$ & $-i$ & $i$ & $i$ & $-i$ & $i$ & $-i$\\
\hline
${\cal R}(\rho, \tau)$ & $i$ & $i$ & $-i$ & $i$ & $-i$ & $i$ & $-i$ & $-i$\\
\hline
\mbox{otherwise}& $1$ & $1$ & $1$ & $1$ & $1$ & $1$ & $1$ & $1$\\
\hline
\end{tabular}
\end{center}
\vspace*{3mm}
\caption{The abelian $3$-cocycles with underlying $3$-cocycle
$\phi_{\{\sigma,\tau\}}$
\talabel{sigmatau}}
\end{table}

\begin{table}[htb]
\begin{center}
\begin{tabular}{c|r|r|r|r|r|r|r|r|}
$EM(\phi_{\{\sigma,\rho\}},\Rr)$
  & $E_2$ & $AE_2$ & $BE_2$ & $CE_2$ & 
 $ABE_2$ & $ACE_2$ & $BCE_2$ & $ABCE_2$ \\
\hline
${\cal R}(\sigma, \sigma)$ & $i$ & $i$ & $i$ & $-i$ & $i$ & $-i$ & $-i$ & $-i$\\
\hline
${\cal R}(\tau, \tau)$ & $1$ & $1$ & $-1$ & $1$ & $-1$ & $1$ & $-1$ & $-1$\\
\hline
${\cal R}(\rho, \rho)$ & $i$ & $-i$ & $i$ & $i$ & $-i$ & $-i$ & $i$ & $-i$\\
\hline
${\cal R}(\sigma, \tau)$ & $1$ & $1$ & $1$ & $1$ & $1$ & $1$ & $1$ & $1$\\
\hline 
${\cal R}(\tau, \sigma)$ & $1$ & $-1$ & $-1$ & $-1$ & $1$ & $1$ & $1$ & $-1$\\
\hline
${\cal R}(\sigma, \rho)$ & $i$ & $i$ & $i$ & $-i$ & $i$ & $-i$ & $-i$ & $-i$\\
\hline
${\cal R}(\rho, \sigma)$ & $i$ & $-i$ & $-i$ & $i$ & $i$ & $-i$ & $-i$ & $i$\\
\hline
${\cal R}(\tau, \rho)$ & $1$ & $-1$ & $1$ & $-1$ & $-1$ & $1$ & $-1$ & $1$\\
\hline
${\cal R}(\rho, \tau)$ & $1$ & $1$ & $-1$ & $1$ & $-1$ & $1$ & $-1$ & $-1$\\
\hline
\mbox{otherwise}& $1$ & $1$ & $1$ & $1$ & $1$ & $1$ & $1$ & $1$\\
\hline
\end{tabular}
\end{center}
\vspace*{3mm}
\caption{The abelian $3$-cocycles with underlying $3$-cocycle
$\phi_{\{\sigma,\rho\}}$
\talabel{sigmarho}}
\end{table}

\begin{table}[htb]
\begin{center}
\begin{tabular}{c|r|r|r|r|r|r|r|r|}
$EM(\phi_{\{\tau,\rho\}},\Rr)$
  & $E_3$ & $AE_3$ & $BE_3$ & $CE_3$ & 
 $ABE_3$ & $ACE_3$ & $BCE_3$ & $ABCE_3$ \\
\hline
${\cal R}(\sigma, \sigma)$ & $1$ & $1$ & $1$ & $-1$ & $1$ & $-1$ & $-1$ & $-1$\\
\hline
${\cal R}(\tau, \tau)$ & $i$ & $i$ & $-i$ & $i$ & $-i$ & $i$ & $-i$ & $-i$\\
\hline
${\cal R}(\rho, \rho)$ & $i$ & $-i$ & $i$ & $i$ & $-i$ & $-i$ & $i$ & $-i$\\
\hline
${\cal R}(\sigma, \tau)$ & $1$ & $1$ & $1$ & $1$ & $1$ & $1$ & $1$ & $1$\\
\hline 
${\cal R}(\tau, \sigma)$ & $1$ & $-1$ & $-1$ & $-1$ & $1$ & $1$ & $1$ & $-1$\\
\hline
${\cal R}(\sigma, \rho)$ & $1$ & $1$ & $1$ & $-1$ & $1$ & $-1$ & $-1$ & $-1$\\
\hline
${\cal R}(\rho, \sigma)$ & $1$ & $-1$ & $-1$ & $1$ & $1$ & $-1$ & $-1$ & $1$\\
\hline
${\cal R}(\tau, \rho)$ & $i$ & $-i$ & $i$ & $-i$ & $-i$ & $i$ & $-i$ & $i$\\
\hline
${\cal R}(\rho, \tau)$ & $i$ & $i$ & $-i$ & $i$ & $-i$ & $i$ & $-i$ & $-i$\\
\hline
\mbox{otherwise}& $1$ & $1$ & $1$ & $1$ & $1$ & $1$ & $1$ & $1$\\
\hline
\end{tabular}
\end{center}
\vspace*{3mm}
\caption{The abelian $3$-cocycles with underlying $3$-cocycle
$\phi_{\{\tau,\rho\}}$
\talabel{taurho}}
\end{table}

\begin{remark}\relabel{added}
This is a continuation of \reref{A16}. It is easy to show that $Q:\ C_2\times C_2\to k^*$
is a quadratic form if and only if $Q(e)=1$ and $Q(\alpha)^4=1$. Assuming that 
$k$ has a primitive fourth root of $1$, we have that $QF(C_2,k^*)=C_4\cong
H^2_{\rm ab}(C_2,k^*)$. The four cocycles in $H^2_{\rm ab}(C_2,k^*)$ are
$$
(1,1),~(1,\Rr_2),~(\phi,\Rr_3),~(\phi,\Rr_4),
$$
with 
$$
\Rr_2=1-2P_\alpha\ot P_\alpha,~
\Rr_3=1-(1-i)P_\alpha\ot P_\alpha,~
\Rr_4=1-(1+i)P_\alpha\ot P_\alpha.
$$
$(1,1)$ and $(1,\Rr_2)$ give symmetries on ${\rm Vect}^{C_2}$; the two other ones
give braided non-symmetric monoidal structures. Now it is easy to calculate that
$$
t_1(1,\Rr_2)=(1,AB),~t_2(1,\Rr_2)=(1,AC),~t_3(1,\Rr_2)=(1,BC),
$$
and these are precisely the nontrivial symmetric abelian cocycles. In a similar way, we can
compute that
\begin{eqnarray*}
t_1(\phi,\Rr_3)=(\phi_{\{\tau,\rho\}}, E_3)&~;~&
t_1(\phi,\Rr_4)=(\phi_{\{\tau,\rho\}}, ABE_3);\\
t_2(\phi,\Rr_3)=(\phi_{\{\sigma,\rho\}}, E_2)&~;~&
t_2(\phi,\Rr_4)=(\phi_{\{\sigma,\rho\}}, ACE_2).
\end{eqnarray*}
$t_3(\phi,\Rr_3)$ and $t_3(\phi,\Rr_4)$ are cohomologous to
$(\phi_{\{\sigma,\tau\}}, E_1)$ and
$(\phi_{\{\sigma,\tau\}}, ABE_1)$.
\end{remark}

\section{Some applications}\selabel{appl} 
\setcounter{equation}{0}
\subsection{Quasi-Hopf algebra structures}\selabel{5.1}
We will examine the following classification problem. Let $H$ be a commutative,
cocommutative, finite dimensional Hopf algebra, in our case let $H=k[C_n]$ or 
$k[C_2\times C_2]$, where $C_n$ is the cyclic group of order $n$. 
Classify, up to gauge equivalence, all quasi-bialgebra structures on $H$. 
Recall briefly that a quasi-bialgebra is a unital 
associative algebra endowed with a coalgebra structure that is coassociative 
up to conjugation by a reassociator $\Phi\in H\ot H\ot H$. For a complete 
definition we invite the reader to consult \cite{k, maj}. 
Notice that a quasi-bialgebra structure on a commutative Hopf algebra $H$ is completely 
determined by such a reassociator, that is a normalized Harrison 3-cocycle $\Phi$ on $H$. 
Furthermore, the quasi-bialgebra $(H, \Phi)$ is gauge equivalent to $(H, 1\ot 1\ot 1)$ if and only if 
$\Phi$ is a coboundary. Since every Harrison 3-cocycle is equivalent to a normalized 
3-cocycle, it follows that our problem is equivalent to computing the third 
Harrison cohomology group $H^3_{\rm Harr}(H,k,\units)$. For the definition and 
generalities on Harrison cohomology, we refer to \cite[Sec. 9.2]{Caenepeel98}.

If $k$ contains a primitive $n$-th root of unit, respectively has characteristic not $2$, 
then the Hopf algebras $k[C_n]$ and $k[C_2\times C_2]$ are isomorphic to their dual Hopf algebras. 
This means that, for $G=C_n$ or $C_2\times C_2$, 
$$
H^3_{\rm Harr}(k[G],k,\units)\cong 
H^3_{\rm Harr}(k[G]^*,k,\units)\cong H^3(G, k^*),
$$
and we are reduced to computing group cohomology.

\begin{proposition}\prlabel{5.1}
Let $C_n=\le c\ri$ be the cyclic group of order $n$ written multiplicatively and 
$k$ a field containing a primitive $n$-th root of unit $\xi$. Then all 
the normalized Harrison $3$-cocycles $\Phi\in k[C_n]\ot k[C_n]\ot k[C_n]$ are of the form 
\[
\Phi_l=1 - \frac{1}{n^2} (1 - c^l)\ot \sum\limits_{i, j=0}^{n-1}(1 -n\d_{i, j})(\xi^j -n\d_{j, 0})c^i\ot c^j,
\]
where $l\in \{0, 1, \cdots, n-1\}$. In particular, if $n=2$, there is a unique non-trivial $3$-cocycle 
$\Phi_1= 1 - 2p_{-}\ot p_{-}\ot p_{-}$, where $p_{-}=\frac{1}{2}(1 - c)$. 
\end{proposition} 

\begin{proof}
As we have mentioned in  the Introduction, we have a bijection between $H^3(C_n, k^*)$ 
and the $n$-th roots of $1$ in $k$. So we have a cocycle for every
positive integer   $l\in\{0, \cdots, n-1\}$. 

Since $k$ contains a primitive $n$-th root of unit $\xi$, we deduce that the characteristic of 
$k$ does not divide $n$ (this follows easily from $n=(1-\xi)(1-\xi^2)\cdots (1-\xi^{n-1})$). Suppose 
$C_n=\le c\ri$, written multiplicatively, and let $\{P_1, P_c, \cdots , P_{c^{n-1}}\}$ be the basis of 
$k[C_n]^*=k^{C_n}$ 
dual to the basis $\{1, c, \cdots, c^{n-1}\}$ of $k[C_n]$. 
Define $f\in k[C_n]^*$ by $f(c^i)=\xi^i$, for all $0\leq i\leq n-1$. Then $f$ is a well defined algebra map 
and $f^j(c^s)=\xi^{js}$, for all $0\leq s\leq n-1$. Furthermore, we know from 
\cite[Exercise 4.3.6]{dnr} that 
\[
\Psi: k[C_n]\ni c^j\mapsto f^j\in k[C_n]^*~,
\]  
extended linearly, is a Hopf algebra isomorphism. Its inverse is defined by  
\[
\Psi^{-1}(P_{c^j})=\frac{1}{n}\sum\limits_{s=0}^{n-1}\xi^{(n-s)j}c^s,~~\forall~~0\leq j\leq n-1.
\]
We easily compute that
\begin{eqnarray*}
&&\hspace*{-1cm}
\sum\limits_{j=0}^{n-1}q^j\Psi^{-1}(P_{c^j})=
\frac{1}{n}\sum\limits_{s=0}^{n-1}\left(\sum\limits_{j=0}^{n-1}\xi^{lj}(\xi^{n-s})^j\right)c^s\\
&&\hspace*{5mm} 
=\frac{1}{n}\sum\limits_{s=0}^{n-1}\left(\sum\limits_{j=0}^{n-1}(\xi^{l - s})^j\right)c^s=
\frac{1}{n}\sum\limits_{j=0}^{n-1}n\d_{l, s}c^s=c^l.
\end{eqnarray*}
Thus, using the definition of $\phi_q$ from \equref{3coccyc} we have that 
any normalized $3$-cocycle $\Phi\in k[C_n]\ot k[C_n]\ot k[C_n]$ is of the form
\begin{eqnarray*}
\Phi_l:&=&\sum\limits_{u, v, s=0}^{n-1}\v_q(c^u, c^v, c^s)\Psi^{-1}(P_{c^u})\ot \Psi^{-1}(P_{c^v})\ot \Psi^{-1}(P_{c^s})\\
&=&
\sum\limits_{u=0}^{n-1}\Psi^{-1}(P_{c^u})\ot \sum\limits_{v+s < n}\Psi^{-1}(P_{c^v})\ot \Psi^{-1}(P_{c^s}) \\
&&\hspace{1.5cm}+ 
\sum\limits_{u=1}^{n-1}q^u\Psi^{-1}(P_{c^u})\ot \sum\limits_{v + s\geq n}\Psi^{-1}(P_{c^v})\ot \Psi^{-1}(P_{c^s})\\
&=&
1\ot \sum\limits_{v + s < n}\Psi^{-1}(P_{c^v})\ot \Psi^{-1}(P_{c^s}) + 
c^l\ot \sum\limits_{v + s\geq n}\Psi^{-1}(P_{c^v})\ot \Psi^{-1}(P_{c^s})\\
&=& 1 - (1 - c^l)\ot \sum\limits_{v + s\geq n}\Psi^{-1}(P_{c^v})\ot \Psi^{-1}(P_{c^s})\\
&=& 1 - (1 - c^l)\ot \sum\limits_{v=1}^{n-1}\sum\limits_{t=0}^{n-2}
\Psi^{-1}(P_{c^v})\ot \Psi^{-1}(P_{c^{n + t - v}})\\
&=& 1 - \frac{1}{n^2}(1 - c^l)\ot \sum\limits_{i, j=0}^{n-1}\sum\limits_{v=1}^{n-1}\sum\limits_{t=0}^{n-2}
\xi^{(n - i)v + (n - j)(n + t - v)}c^i\ot c^j\\
&=& 1 - \frac{1}{n^2}(1 - c^l)\ot 
\sum\limits_{i, j=0}^{n-1}\left(\sum\limits_{v=1}^{n-1}
(\xi^{j - i})^v\right)c^i\ot \left(\sum\limits_{t=0}^{n-2}(\xi^{-j})^t\right)c^j\\
&=& 1 - \frac{1}{n^2} (1 - c^l)\ot \sum\limits_{i, j=0}^{n-1}(1 -n\d_{i, j})(\xi^j -n\d_{j, 0})c^i\ot c^j, 
\end{eqnarray*}
as stated. In the case when $n=2$ we calculate 
\begin{eqnarray*}
&&\hspace*{-1cm}
\sum\limits_{i, j=0}^n (1 -2\d_{i, j})((-1)^j- 2\d_{j, 0})c^i\ot c^j\\
&&\hspace*{1cm}
=1\ot 1 - 1\ot c - c\ot 1 + c\ot c=(1- c)\ot (1 - c),
\end{eqnarray*}
to deduce that $\Phi_1=1 - 2 p_{-}\ot p_{-}\ot p_{-}$, as claimed. Note that it is 
precisely the $3$-cocycle that confers to $k[C_2]$ the unique quasi-bialgebra structure 
that is not twist equivalent to a Hopf algebra, see \cite{eg}. 
\end{proof}

Our next aim is to describe the cocycles in $H^3_{\rm Harr}(k[C_2\times C_2],k,\units)
\cong H^3(C_2\times C_2,k^*)$ more explicitly. We use the notation of the previous
Sections: $C_2\times C_2=\{e,\sigma, \tau, \rho\}$, with $\sigma\tau=\rho$.
There are three 
Hopf algebra maps $k[C_2]\to k[C_2\times C_2]$, so we immediately find three 
Harrison 3-cocycles $\Phi_x= 1-2p_{-}^x\ot p_{-}^x\ot p_{-}^x$, $x=\sigma, \tau, \rho$, 
where $p_{-}^x=\frac{1}{2}(1-x)$.

One of the isomorphisms $\Psi:\ k[C_2\times C_2]^*\cong k[C_2]^*\ot k[C_2]^*
\to k[C_2]\ot k[C_2]\cong k[C_2\times C_2]$ is the following:
\begin{eqnarray*}
\Psi(P_e)=u_e={1\over 4}(1+\sigma+\tau+\rho)&~,~&
\Psi(P_\sigma)=u_\sigma={1\over 4}(1-\sigma+\tau-\rho),\\
\Psi(P_\tau)=u_\tau={1\over 4}(1+\sigma-\tau-\rho)&~,~&
\Psi(P_\rho)=u_\rho={1\over 4}(1-\sigma-\tau+\rho).
\end{eqnarray*}
We can use this isomorphism to write down the Harrison cocycles in the
basis $\{u_e,u_\sigma,u_\tau,u_\rho\}$. It is then possible in principle to write
the cocycles as sums of monomials. Some of the cocycles can be written down
explicitly. Observing that
$$X_\sigma+X_\rho+X_{\rho,\sigma}+X_{\sigma,\rho}=
(P_\sigma+P_\rho)\ot (P_\sigma+P_\rho)\ot (P_\sigma+P_\rho)$$
and
$$\Psi(P_\sigma+P_\rho)=u_\sigma+u_\rho={1\over 2}(1-\sigma),$$
we see that
$$\Psi(\phi_{\{\sigma,\rho\}})=1-2p_-^\sigma\ot p_-^\sigma\ot p_-^\sigma=\Phi_\sigma.$$
In a similar way, we can show that
$\Psi(\phi_{\{\sigma,\rho\}})=\Phi_\tau$ and 
$\Psi(h_{-1}g_{-1}\phi_{\{\sigma,\tau\}})=\Phi_\rho$. If $-1$ has a squareroot in $k$,
then it follows that $\Phi_\rho$ is cohomologous to $\Phi_\sigma\Phi_\tau$. Note that
these observations are consistent with \reref{A16}.

\subsection{Weak braided Hopf algebra structures}\selabel{5.2}
The definition of a weak Hopf algebra can be found in \cite{bos}. For the 
definition of a weak braided Hopf algebra in a symmetric monoidal category 
we refer to \cite{avr, bulacu}. From \cite{bulacu}, we recall the following
construction of weak braided Hopf algebras in ${\rm Vect}^G$.

Let $F:\ G\times G\to k^*$ be a normalized 2-cochain (see \leref{A1}), with pointwise inverse
$F^{-1}$. Consider $\Rr_{F^{-1}}:\ G\times G\to k^*$, 
$\Rr_{F^{-1}}(x,y)=F(x,y)F(y,x)^{-1}$. Then
$(\phi_{F^{-1}}=\Delta_2(F^{-1}), {\cal R}_{F^{-1}})$ is a coboundary abelian
$3$-cocycle on $G$. Let ${\rm Vect}^G_{F^{-1}}$ be category
${\rm Vect}^G$ equipped with the 
braided monoidal structure provided by $(\Delta_2(F^{-1}), {\cal R}_{F^{-1}})$. 

Let $k_F[G]$ be the $k$-vector space $k[G]$ 
with multiplication given by the formula
$$x\bullet y=F(x, y)xy,$$
for all $x, y\in G$. $k_F[G]$ is a commutative algebra in ${\rm Vect}^G_{F^{-1}}$, 
see \cite[Corollary 2.4]{am0}.  

If $|G| \not=0$ in $k$ then
we can define a 
cocommutative coalgebra structure on $k[G]$ in ${\rm Vect}^G_{F^{-1}}$. 
$k^F[G]$, the $k$-vector space $k[G]$ 
with comultiplication and counit given by 
$$
\Delta_F(x)=\frac{1}{|G|} \sum\limits_{u\in G}F(u, u^{-1}x)^{-1}u\ot u^{-1}x
~~{\rm and}~~\va_F(x)=|G| \d_{x, e}~~,
$$
for all $x\in G$, is a cocommutative coalgebra in ${\rm Vect}^G_{F^{-1}}$,
see \cite[Prop. 3.2]{bulacu}. $k_F^F[G]$, the vector space equipped with the
algebra and coalgebra structure defined above, is a commutative and cocommutative
weak braided Hopf algebra in ${\rm Vect}^G_{F^{-1}}$, see 
\cite[Proposition 4.7]{bulacu}. The antipode $\un{S}$ is identity on $k[G]$. 

Examples of such braided Hopf algebras are the Cayley-Dickson and Clifford algebras, 
see \cite{bulacu, bulacu1}. Moreover, they are monoidal Frobenius algebras 
and monoidal co-Frobenius coalgebras in the appropriate braided monoidal category
of graded vector spaces.

We will now apply this construction in the case where $G=C_n$ and
$G=C_2\times C_2$, in order to construct more examples of weak braided Hopf algebras.
We begin with a generalization of \cite[Cor. 10]{am}, where it is shown that the map
$\phi$ in \leref{5.2} is a 3-cocycle and a coboundary in the case where $n=3$.

\begin{lemma}\lelabel{5.2}
Let $C_n=\le\sigma\ri$ be the cyclic group of order $n$ and $k$ a field containing a
primitive $n$-th root of unit $q$. Then $\phi(\sigma^a, \sigma^b, \sigma^c):=q^{abc}$ 
is a normalized $3$-cocycle on $C_n$. $\phi$ is coboundary if and only if 
$q^{\frac{n(n-1)}{2}}=1$. 
\end{lemma}

\begin{proof}
The fact that $q^n=1$ implies that  $\phi$ well defined. Indeed, if $a=a'+n$, 
$b=b'+n$ and $c=c'+n$ then $abc=n^3+ (a'+b'+c')n^2+(a'b'+a'c'+b'c')n+a'b'c'$, and so 
$\phi(\sigma^a, \sigma^b, \sigma^c)=q^{abc}=q^{a'b'c'}=\phi(\sigma^{a'}, \sigma^{b'}, \sigma^{c'})$. 
The $3$-cocycle condition reduces to the $bcd+a(b+c)d+abc=ab(c+d)+(a+b)cd$ in $\mathbb{Z}$,
which is clearly satisfied. It is also clear that $\phi$ is normalized.

If $\phi$ is a coboundary, then there exists $g:\ C_n\times C_n\to k^*$ such that
\begin{equation}\eqlabel{whenphiiscob}
g(\sigma^b, \sigma^c)g(\sigma^{a+b}, \sigma^c)^{-1}g(\sigma^a, \sigma^{b+c})g(\sigma^a, \sigma^b)^{-1}=q^{abc},
\end{equation}
for all $a, b, c\in\mathbb{Z}$. Let $\b :=g(\sigma^a, 1)=g(1, \sigma^b)$ and 
$\a_{c}:=g(\sigma, \sigma^c)$, for all $a, b, c\in \{1, \cdots, n-1\}$.
Taking $a=1$, $b=k$ and $c=n-1$ in \equref{whenphiiscob} we obtain  
$$
g(\sigma^{k+1}, \sigma^{n-1})=q^kg(\sigma^k, \sigma^{n-1})g(\sigma, \sigma^{k-1})g(\sigma, \sigma^k)^{-1},
$$
for all $k\in\mathbb{Z}$. 
By mathematical induction it follows that 
$$
g(\sigma^k, \sigma^{n-1})=q^{\frac{k(k-1)}{2}}\a_{n-1}\b\a_{k-1}^{-1},
$$ 
for all $2\leq k\leq n-1$. We then have 
\begin{eqnarray*}
\b&=&g(1, \sigma^{n-1})=g(\sigma^{(n-1)+1}, \sigma^{n-1})\\
&=&q^{n-1}g(\sigma^{n-1}, \sigma^{n-1})g(\sigma, \sigma^{n-2})
g(\sigma, \sigma^{n-1})^{-1}\\
&=&q^{n-1}q^{\frac{(n-1)(n-2)}{2}}\a_{n-1}\b\a_{n-2}^{-1}\a_{n-2}\a_{n-1}^{-1}=
q^{\frac{n(n-1)}{2}}\b, 
\end{eqnarray*}
and we conclude that $q^{\frac{n(n-1)}{2}}=1$.

Conversely, assume that $q^{\frac{n(n-1)}{2}}=1$ and consider 
$f: \mathbb{Z}\times \mathbb{Z}\ra \mathbb{Z}$, $f(x, y)=-\frac{(x-1)xy}{2}$ 
and $g: C_n\times C_n\ra k$, $g(\sigma^a, \sigma^b)=q^{f(a, b)}$. 
If $a=a'+n$ and $b=b'+n$ then it can be easily checked that 
$$
f(a, b) - f(a', b')=-a'n^2-\left(\frac{a'(a'-1)}{2} + a'b'\right)n - \frac{n^2(n-1)}{2}-\frac{n(n-1)}{2}b'.
$$
This relation together with $q^n=1$ and $q^{\frac{n(n-1)}{2}}=1$ implies that $g$ is well defined. 
A straightforward computation now shows that
$$
f(y, z) - f(x+y, z) + f(x, y+z) - f(x, y)=xyz,
$$
for all $x, y, z\in\mathbb{Z}$, proving that $\Delta_2(g)=\phi$, and $\phi$ is a coboundary.
\end{proof}

\begin{proposition}\prlabel{5.3}
Let $k$ be a field containing a primitive $n$-th root of unit $q$ such that $q^{\frac{n(n-1)}{2}}=1$ 
(consider for instance $n$ odd). If $C_n=\le \sigma\ri$ is the cyclic group of order 
$n$ generated by $\sigma$ and $F: C_n\times C_n\ra k^*$ is given by 
$F(\sigma^a, \sigma^b)=q^{-\frac{(a-1)ab}{2}}$, for all 
$0\leq a, b\leq n-1$, then $k[C_n]$ is a commutative and cocommutative 
weak braided Hopf algebra in ${\rm Vect}^{C_n}_{F^{-1}}$ via the structure 
$$\sigma^a\bullet \sigma^b=q^{-\frac{(a-1)ab}{2}}\sigma^{a + b}~~,~~
\un{\Delta}(\sigma^a)=\frac{1}{n}\sum\limits_{l=0}^{n-1}q^{(l-1)l(a-l)}\sigma^l\ot \sigma^{a-l},$$
for all $0\leq a, b\leq n-1$. The unit is $e$ and the counit is given by 
$\un{\va}(\sigma^a)=n\d_{a, 0}$, for all $0\leq a\leq n-1$; the antipode is the 
identity on $k[C_n]$. 
\end{proposition} 

\begin{proof}
This follows immediately from \leref{5.2} and the general construction of $k_F^F[G]$ 
presented above. We leave the verification of the details to the reader. 
\end{proof}

The 2-cochains discussed in \prref{A12} can also be applied to construct examples of weak
braided Hopf algebras.

\begin{proposition}\prlabel{5.4}
On $k[C_2\times C_2]$ we have the following commutative and 
cocommutative weak braided Hopf algebra structure:
\begin{itemize}
\item[(i)] The multiplication table is as follows ($a\in k^*$ is a fixed scalar):
\[
\begin{tabular}{c|r|c|c|c|}
$\bullet$ & $e$ & $\sigma$ & $\tau$ & $\rho$ \\
\hline
$e$ & $e$ & $\sigma$ & $\tau$ & $\rho$\\
\hline
$\sigma$ & $\sigma$ & $a^{-1}e$ & $\rho$ & $\tau$ \\
\hline
$\tau$ & $\tau$ & $\rho$ & $a^{-1}e$ & $\sigma$ \\
\hline
$\rho$ & $\rho$ & $\tau$ & $\sigma$ & $a^{-1}e$ \\
\hline
\end{tabular}~~, 
\]
The comultiplication  is given by the formulas
\begin{eqnarray*}
&&\Delta(e)=\frac{1}{4}\left(e\ot e + a\sigma\ot\sigma + a\tau\ot\tau + a\rho\ot \rho\right)~,\\
&&\Delta(\sigma)=\frac{1}{4}\left( e\ot\sigma + \sigma\ot e + \tau\ot \rho + \rho\ot\tau\right)~,\\
&&\Delta(\tau)=\frac{1}{4}\left(e\ot \tau + \tau\ot e + \sigma\ot \rho + \rho\ot \sigma\right)~,\\
&&\Delta(\rho)=\frac{1}{4}\left(e\ot \rho + \rho\ot e + \sigma\ot\tau + \tau\ot \sigma\right)~,
\end{eqnarray*}
while the counit is given by $\va(e)=4$ and $\va(x)=0$, for $x\in \{\sigma, \tau, \rho\}$. 
The antipode is the identity on $k[C_2\times C_2]$. 

\item[(ii)] For any $d\in k^*$ the multiplication table
\[
\begin{tabular}{c|r|c|c|c|}
$\bullet$ & $e$ & $\sigma$ & $\tau$ & $\rho$ \\
\hline
$e$ & $e$ & $\sigma$ & $\tau$ & $\rho$\\
\hline
$\sigma$ & $\sigma$ & $d^{-1}e$ & $\rho$ & $d^{-1}\tau$ \\
\hline
$\tau$ & $\tau$ & $d^{-1}\rho$ & $d^{-1}e$ & $\sigma$ \\
\hline
$\rho$ & $\rho$ & $\tau$ & $d^{-1}\sigma$ & $d^{-1}e$ \\
\hline
\end{tabular}~~, 
\]
 the comultiplication given by the formulas
\begin{eqnarray*}
&&\Delta(e)=\frac{1}{4}\left(e\ot e + d\sigma\ot\sigma + d\tau\ot\tau + d\rho\ot \rho\right)~,\\
&&\Delta(\sigma)=\frac{1}{4}\left( e\ot\sigma + \sigma\ot e + \tau\ot \rho + d\rho\ot\tau\right)~,\\
&&\Delta(\tau)=\frac{1}{4}\left(e\ot \tau + \tau\ot e + d\sigma\ot \rho + \rho\ot \sigma\right)~,\\
&&\Delta(\rho)=\frac{1}{4}\left(e\ot \rho + \rho\ot e + \sigma\ot\tau + d\tau\ot \sigma\right)~,
\end{eqnarray*}
and counit $\va(e)=4$ and $\va(x)=0$, for all $x\in \{\sigma, \tau, \rho\}$, 
makes $k[C_2\times C_2]$ a weak braided Hopf algebra. 
The antipode is the identity on $k[C_2\times C_2]$.
\end{itemize}
\end{proposition}

\begin{proof}
(i) We use the fact that $h_a$ is a coboundary, for all $a\in k^*$. Actually, 
if we take $g: C_2\times C_2\ra k^*$ defined by $b_i=1$, for all $1\leq i\leq 6$, 
$a_1=a_2=a_3=a$ and $c=1$ then we have seen that $\Delta_2(g)=h_a$. Consequently, 
$g$ is a $2$-cochain on $C_2$. A simple inspection shows us that the claimed structure 
on $k[C_2\times C_2]$ from (i) coincides with that of $k_{g^{-1}}^{g^{-1}}[C_2\times C_2]$, so 
we are done. Note that we have a weak braided Hopf algebra in ${\rm Vect}^{C_2\times C_2}_{g}$, 
and that $\Delta_2(g)=h_a$, while ${\cal R}_g(x, y)=1$, for all $x, y\in C_2\times C_2$. 

(ii) We proceed as above, but now we take $b=d^2$, for some $d\in k^*$. Then $g_{d}$ 
is coboundary, cf. \prref{A12}. More precisely, if $g: C_2\times C_2\ra k^*$ 
is defined by $a_1=a_2=a_3=b_4=b_5=b_6=d$ and $b_1=b_2=b_3=c=1$ then $\Delta_2(g)=g_b=g_{d^2}$. 
We leave to the reader to check that the second weak braided Hopf algebra structure 
from the statement is precisely the one on $k_{g^{-1}}^{g^{-1}}[C_2\times C_2]$. We notice only 
that in this case we have the braided monoidal structure on ${\rm Vect}^{C_2\times C_2}$ 
produced by $g$, that is $\Delta_2(g)=g_{d^2}$ and 
\begin{eqnarray*}
&&{\cal R}_g(x, x)={\cal R}_g(e, y)={\cal R}_g(z, e)=1~~,~~\forall~~x, y, z\in C_2\times C_2~~,\\
&&{\cal R}_g(\sigma, \tau)={\cal R}_g(\tau, \rho)={\cal R}_g(\rho, \sigma)=d~~{\rm and}\\
&&{\cal R}_g(\tau, \sigma)={\cal R}_g(\sigma, \rho)={\cal R}_g(\rho, \tau)=d^{-1}~.
\end{eqnarray*}  
\end{proof}

 
\end{document}